\documentclass[11pt,a4paper]{amsart}
\usepackage{amsmath,amsthm,amsfonts,amssymb,graphicx,psfrag,dsfont}
\usepackage{color}

\usepackage{hyperref}
\usepackage[alphabetic,initials]{amsrefs}

\psfrag{M-}{$M_-$}
\psfrag{M+}{$M_+$}
\psfrag{What}{$\overline{W}$}
\psfrag{RtimesM+}{$\RR \times M_+$}
\psfrag{RtimesM-}{$\RR \times M_-$}
\psfrag{v0}{$v_0$}
\psfrag{v1-}{$v_1^-$}
\psfrag{v2-}{$v_2^-$}
\psfrag{v3-}{$v_3^-$}
\psfrag{v1+}{$v_1^+$}
\psfrag{g1+}{$\gamma_1^+$}
\psfrag{g1-}{$\gamma_1^-$}
\psfrag{g2-}{$\gamma_2^-$}
\psfrag{uk}{$u_k$}
\psfrag{B4}{$B^4$}
\psfrag{[0,1]xS3}{$[0,1] \times S^3$}
\psfrag{L}{$L$}
\psfrag{UL}{$\uU_L$}
\psfrag{Lambda}{$\Lambda$}
\psfrag{W+}{$W_+$}
\psfrag{uinfty+}{$u_\infty^+$}
\psfrag{uinfty-}{$u_\infty^-$}
\psfrag{gamma}{$\gamma$}
\psfrag{S2n-1}{$S^{2n-1}$}
\psfrag{uinfty}{$u_\infty$}
\psfrag{M}{$M$}
\psfrag{ell}{$\ell$}
\psfrag{RtimesM}{$\RR \times M$}

\theoremstyle{plain}

\newtheorem{prop}{Proposition}[section]
\newtheorem{thm}[prop]{Theorem}

\newtheorem{cor}[prop]{Corollary}

\newtheorem{lem}[prop]{Lemma}

\theoremstyle{definition}
\newtheorem{exmp}[prop]{Example}

\newtheorem{defn}[prop]{Definition}

\newtheorem*{convention}{Convention}

\theoremstyle{remark}
\newtheorem{remark}[prop]{Remark}

\numberwithin{equation}{section}

\newcommand{\defin}{\textbf}

\newcommand{\ev}{\operatorname{ev}}

\newcommand{\Hom}{\operatorname{Hom}}

\newcommand{\Lie}{{\mathcal{L}}}
\newcommand{\loc}{\operatorname{loc}}

\newcommand{\muCZ}{\mu_{\operatorname{CZ}}}
\newcommand{\nice}{^{\operatorname{nice}}}

\newcommand{\std}{_{\operatorname{std}}}

\newcommand{\wind}{\operatorname{wind}}
\newcommand{\windpi}{\operatorname{wind}_\pi}

\newcommand{\U}{\operatorname{U}}

\newcommand{\selflinking}{\operatorname{sl}}

\newcommand{\CC}{{\mathbb C}}
\newcommand{\DD}{{\mathbb D}}

\newcommand{\NN}{{\mathbb N}}

\newcommand{\RR}{{\mathbb R}}

\newcommand{\ZZ}{{\mathbb Z}}

\newcommand{\dD}{{\mathcal D}}

\newcommand{\jJ}{{\mathcal J}}

\newcommand{\mM}{{\mathcal M}}

\newcommand{\uU}{{\mathcal U}}
\newcommand{\vV}{{\mathcal V}}
\newcommand{\1}{\mathds{1}}
\newcommand{\p}{\partial}
\renewcommand{\dbar}{\bar{\partial}}
\newcommand{\Cinftyloc}{\mc{C}^\infty_{\loc}}

\definecolor{blue}{rgb}{0,0,1}
\definecolor{red}{rgb}{1,0,0}
\definecolor{green}{rgb}{0,.7,0}

\definecolor{Chris}{rgb}{1,0,1}

\renewcommand{\a}{\alpha}

\renewcommand{\c}{\gamma}
\renewcommand{\d}{\delta}
\newcommand{\dtotal}{\delta_{\operatorname{total}}}
\newcommand{\dinfty}{\delta_{\i}}

\newcommand{\op}[1]{\operatorname{#1}}

\newcommand{\mc}{\mathcal}
\newcommand{\mb}{\mathbb}
\renewcommand{\i}{\infty}

\renewcommand{\.}{\cdot}

\renewcommand{\o}{\omega}

\newcommand{\ind}[1]{\operatorname{ind}(#1)}
\newcommand{\cz}[2]{\mu_{CZ}^{#2}(#1)}

\title[Unknotted orbits and nicely embedded curves]{Unknotted Reeb orbits
and nicely embedded holomorphic curves}
\author{Alexandru Cioba}
\address{Department of Mathematics \\ 
University College London \\
Gower Street \\
London WC1E 6BT \\ 
United Kingdom}
\email{a.cioba.12@ucl.ac.uk}
\author{Chris Wendl}
\address{Institut f\"ur Mathematik\\
Humboldt-Universit\"at zu Berlin\\
Unter den Linden~6\\
10099 Berlin\\
Germany}
\email{wendl@math.hu-berlin.de}

\begin{document}

\begin{abstract}
We exhibit a distinctly low-dimensional dynamical obstruction to the existence 
of Liouville cobordisms: for any contact $3$-manifold admitting an exact 
symplectic cobordism to the tight $3$-sphere, every nondegenerate contact
form admits an embedded Reeb orbit that is unknotted and has self-linking
number~$-1$.  The same is true moreover for any contact structure on a 
closed $3$-manifold that is reducible.  Our results generalize 
an earlier theorem of Hofer-Wysocki-Zehnder for the $3$-sphere,
but use somewhat newer techniques:
the main idea is to exploit the intersection theory of punctured holomorphic
curves in order to understand the compactification of the space of
so-called ``nicely embedded'' curves in symplectic cobordisms.  In the
process, we prove a local adjunction formula for holomorphic annuli
breaking along a Reeb orbit, which may be of independent interest.
\end{abstract}

\maketitle

\tableofcontents

\section{Introduction}
\label{sec:intro}

\subsection{Statement of the main results}
\label{sec:statement}

Contact structures arise in the context of Hamiltonian dynamics 
via the notion of \emph{convexity}: a convex hypersurface in a symplectic
manifold naturally inherits a contact structure, and the orbits of its
Reeb vector field then match the Hamiltonian orbits defined by any
Hamiltonian function that has the hypersurface as a regular level set.
In this paper, we consider
contact structures that are induced on convex and concave 
\emph{boundaries} of symplectic manifolds, i.e.~symplectic cobordisms.
Our main theorem relates the existence of exact symplectic cobordisms
between given contact manifolds
to a dynamical condition on their Reeb vector fields.  
In particular, we will restrict attention to dimension three
and discuss the existence of closed Reeb orbits $\gamma : S^1 \to M$
that are not only contractible but also \defin{unknotted}, meaning
$$
\gamma = f|_{\p\DD^2} \quad\text{ for some embedding 
$f : \DD^2 \hookrightarrow M$},
$$
where $\DD^2 \subset \CC$ denotes the closed unit disk.
All definitions relevant to the following statements may be found 
in~\S\ref{sec:definitions}, {but let us stress the following
convention from the start since it sometimes causes confusion:}

\begin{convention}
In this paper, the words ``symplectic cobordism \emph{from} $(M_1,\xi_1)$ 
\emph{to} $(M_2,\xi_2)$'' always mean that $(M_1,\xi_1)$ is the \emph{concave}
boundary and $(M_2,\xi_2)$ the \emph{convex} boundary of the cobordism
(cf.~\S\ref{sec:definitions}).  This usage is standard, and is natural
from the perspective of contact surgery, but a few other authors
(especially e.g.~in the literature on embedded contact homology)
sometimes interchange the order of ``convex'' and ``concave,'' which would make
our results false.
\end{convention}

\begin{thm}
\label{thm:main}
Assume $(M,\xi)$ is a closed contact $3$-manifold that admits a Liouville
cobordism to the standard contact $3$-sphere $(S^3,\xi\std)$.  
Then for every nondegenerate contact form
$\alpha$ on $(M,\xi)$, the Reeb vector field $R_\alpha$ admits a simple
closed orbit $\gamma$ whose image is the boundary of an embedded disk
$\dD \subset M$.  Moreover, the Conley-Zehnder index and
self-linking number of $\gamma$ with respect to~$\dD$ satisfy
$$
\muCZ(\gamma ; \dD) \in \{2,3\} \quad \text{ and } \quad
\selflinking(\gamma ; \dD) = -1.
$$
\end{thm}

A minor variation on the same techniques in the spirit of
\cite{Hofer:weinstein} will also imply the following:

\begin{thm}
\label{thm:nonprime}
Assume $(M,\xi)$ is a closed contact $3$-manifold and that either of the
following is true:
\begin{enumerate}
\item \label{item:nonprime}
$M$ is reducible, i.e.~it contains an embedded $2$-sphere that does
not bound an embedded ball;
\item \label{item:OT}
$(M,\xi)$ admits a Liouville cobordism to an overtwisted contact manifold.
\end{enumerate}
Then for every nondegenerate contact form
$\alpha$ on $(M,\xi)$, the Reeb vector field $R_\alpha$ admits a simple
closed orbit $\gamma$ whose image is the boundary of an embedded disk
$\dD \subset M$ such that
$$
\muCZ(\gamma ; \dD) = 2 \quad \text{ and } \quad
\selflinking(\gamma ; \dD) = -1.
$$
\end{thm}

Recall that an oriented $3$-manifold is reducible if and only if
it is either $S^1 \times S^2$ or $M_1 \# M_2$ for a pair of closed
oriented $3$-manifolds that are not spheres.  This condition is now 
known to be equivalent 
to the hypothesis $\pi_2(M) \ne 0$ used in \cite{Hofer:weinstein}:
in one direction this follows from the sphere theorem for $3$-manifolds,
and in the other, from \cite{Hatcher:3manifolds}*{Prop.~3.10} and the
Poincar\'e conjecture.  Note that both of the above theorems require 
nondegeneracy of the contact form $\alpha$, but it is possible for the sake of 
applications to weaken this condition; see Theorem~\ref{thm:technical} below.

\subsection{Context}
\label{sec:context}

The prototype for Theorems~\ref{thm:main} and~\ref{thm:nonprime} is a
20-year-old result of Hofer-Wysocki-Zehnder \cite{HWZ:unknotted},
which amounts to the case $(M,\xi) = (S^3,\xi\std)$ of
Theorem~\ref{thm:main}.  The result in \cite{HWZ:unknotted} was in some
sense far ahead of its time, as it required ideas from both the compactness
theory \cite{SFTcompactness} and the intersection theory 
\cite{Siefring:intersection} of punctured holomorphic curves, but it
appeared several years before either of those theories were developed in
earnest.  In the mean time the available techniques have improved, and our
proofs will make use of those improvements.

A weaker version of Theorem~\ref{thm:main} can be shown to hold in all 
dimensions, namely:

\begin{thm}
\label{thm:allDimensions}
If $(M,\xi)$ is a closed $(2n-1)$-dimensional contact manifold
admitting a Liouville cobordism to a standard contact sphere
$(S^{2n-1},\xi\std)$, then every contact form for $(M,\xi)$ admits a
contractible closed Reeb orbit.
\end{thm}
This result can largely be attributed to Hofer, as most of the ideas
needed for its proof are present in \cite{Hofer:weinstein}.
We will sketch a proof in \S\ref{sec:proof} which is similar in spirit to one
that has previously appeared in the work of Geiges and Zehmisch 
\cite{GeigesZehmisch:cobordisms}*{Corollary~3.3} (see also \cites{GeigesZehmisch:4ball,GeigesZehmisch:4ballErratum});
there is also an alternative proof via symplectic homology 
by Albers, Cieliebak and Oancea (see the appendix of~\cite{CieliebakOancea:EilenbergSteenrod}).
Analogous results that may be viewed as higher-dimensional versions of
Theorem~\ref{thm:nonprime} have appeared in 
\cites{AlbersHofer,NiederkruegerRechtman,GeigesZehmisch:connSum,GhigginiNiederkruegerWendl:subcritical}.
The conclusions of our main results however are stronger and uniquely
low dimensional: for instance in \S\ref{sec:application} below, we will see 
examples of contact $3$-manifolds
that always admit contractible but not necessarily unknotted Reeb orbits.
Theorem~\ref{thm:main} thus gives a new means of proving that these examples
cannot be exactly cobordant to the standard $3$-sphere.

We are aware of three general classes of contact $3$-manifolds that satisfy
the hypothesis of Theorem~\ref{thm:main}.  

\begin{exmp}
If $\xi$ is overtwisted, then a theorem of Etnyre and Honda
\cite{EtnyreHonda:cobordisms} provides Stein cobordisms from
$(M,\xi)$ to any other contact $3$-manifold, so in particular 
to~$(S^3,\xi\std)$.  Of course, in this case Theorem~\ref{thm:nonprime}
also applies and gives a slightly stronger result.
\end{exmp}

\begin{exmp}
Suppose $(M,\xi)$ is subcritically Stein fillable, or equivalently, 
that it can be obtained by performing contact connected sums on copies of
the tight $S^3$ and $S^1 \times S^2$.  In this case, $(M,\xi)$ is the convex
boundary of a Weinstein domain $W$ constructed by attaching $1$-handles to
a ball, and these $1$-handles can then be cancelled by attaching suitable
Weinstein $2$-handles.  This procedure embeds $W$ into the standard $4$-ball 
as a Weinstein subdomain and thus produces a Weinstein cobordism from
$(M,\xi)$ to~$(S^3,\xi\std)$.  Note that Theorem~\ref{thm:nonprime}
also applies in this case unless $M = S^3$.
\end{exmp}

The third class of examples was brought to our attention by Emmy Murphy.

\begin{exmp}
\label{ex:cap}
Suppose $L \subset [1,\infty) \times S^3$ is an \emph{exact Lagrangian cap} for 
some Legendrian knot $\Lambda$ in $(S^3,\xi\std)$, i.e.~$L$ is a compact 
Lagrangian submanifold properly embedded in the top half of the symplectization
$\RR \times S^3$, such that $\p L = \{1\} \times \Lambda$, $L$ is tangent
near its boundary to a globally defined 
Liouville vector field pointing transversely inward
at $\{1\} \times S^3$, and the restriction of the corresponding 
Liouville form to $L$ is
exact.  A result of Francesco Lin \cite{Lin:caps} guarantees
that such caps always exist after stabilizing $\Lambda$ sufficiently many
times.  Now suppose $\uU_L$ is an open neighbourhood of
$L$ in $[1,\infty) \times S^3$, where the latter is viewed as sitting on top
of the standard Weinstein filling $B^4$ of $(S^3,\xi\std)$.  This
neighbourhood can be choosen such that, after smoothing
corners, $B^4 \cup \overline{\uU}_L$ is a Weinstein filling of some contact 
$3$-manifold $(M,\xi)$, and $\left([1,T] \times S^3\right) \setminus \uU_L$
for suitable $T > 1$ defines a Liouville cobordism $W_+$ from $(M,\xi)$ to $(S^3,\xi\std)$,
see Figure~\ref{fig:cap}.  Using a Morse function on $L$ that has one
index~$2$ critical point and an inward gradient at~$\p L$, one can find
a Weinstein handle decomposition of $B^4 \cup \overline{\uU}_L$ having exactly one
$2$-handle (see Remark~\ref{remark:subcritical}), thus $B^4 \cup \overline{\uU}_L$ is 
not subcritical, and it follows from
the uniqueness of Stein fillings in the subcritical case
\cite{CieliebakEliashberg}*{Theorem~16.9(c)} that $(M,\xi)$ is not 
subcritically fillable.  For more details on this construction, 
see Appendix~\ref{app:cap}.

One can now use a well-known result of Eliashberg 
\cites{Eliashberg:diskFilling,CieliebakEliashberg} to extract from 
this example contact $3$-manifolds other than $(S^3,\xi\std)$ 
to which Theorem~\ref{thm:main} applies but
Theorem~\ref{thm:nonprime} does not.  Indeed, while 
$(M,\xi) = \p (B^4 \cup \overline{\uU}_L)$ could be reducible, 
it is Stein fillable and therefore tight, so Colin \cite{Colin:prime}
(see also \cite{Geiges:book}*{\S 4.12}) provides a prime decomposition 
$$
(M,\xi) = (M_1,\xi_1) \# \ldots \# (M_k,\xi_k),
$$
and Eliashberg's theorem implies that $B^4 \cup \overline{\uU}_L$ must be
Weinstein deformation equivalent to a domain obtained by attaching
Weinstein $1$-handles to Weinstein fillings of the summands.  But the summands
cannot all be $S^1 \times S^2$ since $(M,\xi)$ is not subcritically
fillable, so at least one of them is an irreducible tight contact $3$-manifold
admitting a Liouville cobordism to $(S^3,\xi\std)$.
\end{exmp}

\begin{figure}
\includegraphics{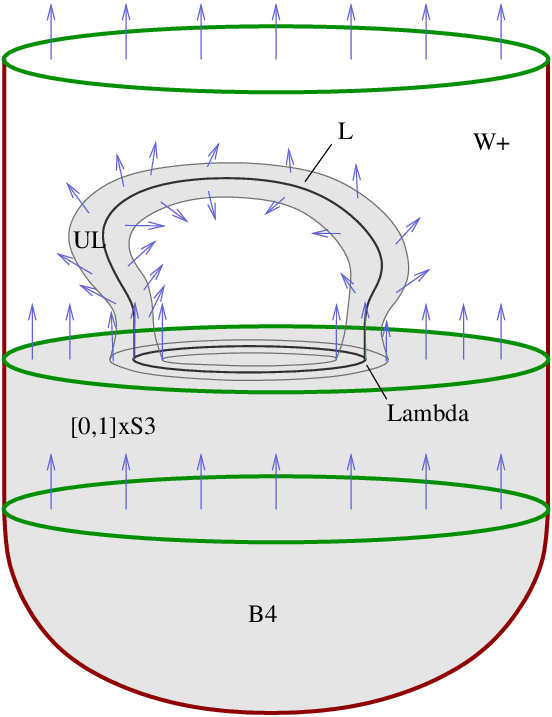}
\caption{\label{fig:cap} An exact Lagrangian cap for a Legendrian in 
$(S^3,\xi\std)$ produces a Liouville cobordism $W_+$ from $(M,\xi)$ to
$(S^3,\xi\std)$, where $(M,\xi) := \p(B^4 \cup \overline{\uU}_L)$ is not subcritically
fillable.}
\end{figure}

\begin{cor}
The contact $3$-manifolds $(M,\xi)$ described in Example~\ref{ex:cap}
and their prime summands all
admit unknotted Reeb orbits with Conley-Zehnder index $2$ or $3$ and
self-linking number $-1$ for every choice of nondegenerate contact form.
\end{cor}

The construction outlined in Example~\ref{ex:cap} also works in higher
dimensions using the exact Lagrangian caps of Eliashberg-Murphy
\cite{EliashbergMurphy:caps}, cf.~Appendix~\ref{app:cap}.
In this case it produces Weinstein
subdomains of the standard ball which are presumably \emph{flexible} in the 
sense of \cite{CieliebakEliashberg}.  Recently, Murphy and Siegel 
\cite{MurphySiegel:subflexible} have also found examples of nonflexible 
Weinstein subdomains in the standard ball, whose boundaries therefore also
satisfy the hypothesis of Theorem~\ref{thm:allDimensions}.

\begin{remark}
It is not known whether any contact $3$-manifolds satisfy the hypothesis
of Theorem~\ref{thm:nonprime}\eqref{item:OT} without being overtwisted, though
Andy Wand \cite{Wand:surgery} has proved that the answer is no under the 
stronger condition that the cobordism is Stein.  
Theorem~\ref{thm:nonprime}\eqref{item:OT} may thus
be interpreted as a small measure of support for the conjecture that Wand's
theorem extends to Liouville cobordisms
(cf.~\cite{Wendl:blogSurvey4}*{Question~5}): that is,
Theorem~\ref{thm:nonprime}\eqref{item:OT} provides a mechanism for detecting
tightness, but it cannot detect the (conjecturally nonexistent) 
distinction between an overtwisted contact manifold and one that is only
Liouville cobordant to something overtwisted.
\end{remark}

We remark that the word ``Liouville'' definitely cannot be
dropped from the statements of any of the above theorems:
for instance, any Lagrangian torus in the standard
symplectic $\RR^{2n}$ gives rise to a strong symplectic cobordism from
the unit cotangent bundle of the torus to $(S^{2n-1},\xi\std)$, but one
can easily find contact forms on the former that have no contractible
Reeb orbits, corresponding to metrics on the torus with no contractible
geodesics.  The cobordism of course cannot be Liouville because,
by a well-known theorem of Gromov \cite{Gromov}, the Lagrangian torus
cannot be exact.  Similarly, \cite{Gay:GirouxTorsion} and \cite{Wendl:cobordisms}
show that every contact $3$-manifold with positive Giroux torsion is
strongly symplectically cobordant to something overtwisted, including e.g.~the nonfillable
tight $3$-tori, which admit contact forms without contractible orbits.

\subsection{Applications}
\label{sec:application}

Here is a specific situation in which Theorem~\ref{thm:main} can be used to 
rule out the existence of exact symplectic cobordisms.  Good candidates for 
manifolds that fail to satisfy the conclusion of the theorem are furnished
by the universally tight lens spaces $L(p,q)$ for $p \neq 1$. 
Recall that $L(p,q)$ is defined as the quotient
$$
L(p,q) = S^3 \big/ G_{p,q},
$$
where $G_{p,q} \subset \U(2)$ denotes the cyclic group of matrices 
$\begin{pmatrix} e^{2\pi i k / p} & 0 \\ 0 & e^{2\pi i k q / p} \end{pmatrix}$
for $k \in \ZZ_p$, acting on the unit sphere $S^3 \subset \CC^2$ by unitary
transformations.  This action preserves the standard contact form
$\alpha\std = \frac{1}{2} \sum_{j=1}^2 (x_j \, d y_j - y_j \, d x_j)$
on~$S^3$, written here in coordinates $(z_1,z_2) = (x_1 + i y_1,x_2 + i y_2)$,
so the standard contact structure $\xi\std$ on $L(p,q)$ is defined via this
quotient.

\begin{prop}
\label{prop:lens}
For every relatively prime pair of integers $p > q \ge 1$,
$L(p,q)$ admits a nondegenerate contact form with only two simple closed
Reeb orbits, both of them nondegenerate and noncontractible.
\end{prop}
\begin{proof}
We present $(L(p,q),\xi\std)$ as a quotient of the so-called
\emph{irrational ellipsoid}.
Let $\alpha_H := \frac{1}{H} \alpha\std$ on $S^3$, where $H$ is the restriction
to the unit sphere $S^3 \subset \CC^2$ of the function
$$
H(z_1,z_2) = \frac{|z_1|^2}{a^2} + \frac{|z_2|^2}{b^2}
$$
for some $a , b > 0$.  The closed orbits for the Reeb flow on $S^3$ determined 
by  $\alpha_H$ are then in bijective correspondence with the closed 
orbits on the ellipsoid $H^{-1}(1) \subset \CC^2$ for the Hamiltonian flow of
$H$ on the standard symplectic~$\CC^2$.  In particular, if $a / b$ 
is irrational, then the only simple closed orbits of this flow are
(up to parametrization) the embedded loops $\gamma_1, \gamma_2 : S^1 \to S^3
\subset \CC^2$ defined by
$$
\gamma_1(t) = (e^{2\pi it},0), \qquad \gamma_2(t) = (0,e^{2\pi it})
$$
for $t \in S^1 = \RR / \ZZ$, and moreover, these orbits and their multiple
covers are all nondegenerate.  Now since $\alpha\std$ and $H$ are both 
invariant under the action of $\U(1) \times \U(1) \subset \U(2)$, which 
contains $G_{p,q}$, $\alpha_H$ descends to a well-defined contact form on 
$L(p,q)$, and this contact form is nondegenerate.  But the orbits $\gamma_1$ 
and $\gamma_2$
project to orbits in $L(p,q)$ that are $p$-fold covered, so their underlying
simple orbits lift to the universal cover $S^3 \to L(p,q)$ as non-closed
paths since $p > 1$, hence they are noncontractible.
\end{proof}

\begin{cor}
For every pair of relatively prime integers $p > q \ge 1$,
$(L(p,q), \xi\std)$ admits no exact cobordism to $(S^3,\xi\std)$. 
\end{cor}

\begin{remark}
The Reeb flow on any universally tight $L(p,q)$ admits a contractible Reeb 
orbit since $\pi_1(L(p,q))$ is torsion, so 
previously known criteria for excluding such cobordisms do not apply.
\end{remark}

While the lens space example is relatively easy to work with, the 
nondegeneracy of a contact form is usually a rather difficult condition to 
check, and for this reason one might sometimes want to have the following
technical enhancement of Theorems~\ref{thm:main} and~\ref{thm:nonprime}.
It will be an immediate consequence of our proofs, requiring only that one
pay closer attention to the
relationship between periods of orbits and energies of holomorphic curves.

\begin{thm}
\label{thm:technical}
Assume $(M,\xi)$ satisfies the hypotheses of either Theorem~\ref{thm:main}
or Theorem~\ref{thm:nonprime}, and fix a contact form $\alpha_0$ for
$(M,\xi)$.  There exists a constant $T > 0$, dependent on~$\alpha_0$,
such that the following holds: suppose $\alpha = f \alpha_0$ is a contact
form on $(M,\xi)$ such that
\begin{enumerate}
\item $f : M \to (0,\infty)$ satisfies $f < T$, and
\item All closed Reeb orbits for $\alpha$ with period less than $T$
are nondegenerate.
\end{enumerate}
Then the Reeb flow of $\alpha$ satisfies the conclusions of 
Theorems~\ref{thm:main} or~\ref{thm:nonprime} respectively, and the 
unknotted orbit can be assumed to have period less than~$T$.
\end{thm}

One could apply this in practice if e.g.~$\alpha_0$ is Morse-Bott and admits
no unknotted Reeb orbits, as then one can define perturbations of $\alpha_0$
as in \cite{Bourgeois:thesis} whose orbits up to some arbitrarily large
period are nondegenerate and still knotted---the topology of orbits with
large period may be harder to control, but for Theorem~\ref{thm:technical}
this does not matter.

\begin{remark}
We have chosen to adopt a mainly contact topological perspective on the
main theorems of this paper, but for other purposes 
(e.g.~quantitative Reeb dynamics, cf.~\cite{GeigesZehmisch:4ball}*{\S 3.23}),
one could also state
more quantitatively precise versions of Theorem~\ref{thm:technical}.
\end{remark}

Note that no such enhancement is necessary for Theorem~\ref{thm:allDimensions},
which does not require nondegeneracy, see Remark~\ref{remark:degenerate}.

\subsection{Outline of proofs, part~1: seed curves and compactness}
\label{sec:proof}

All proofs of theorems in this paper follow a similar scheme, which in the
case of Theorems~\ref{thm:main} and~\ref{thm:allDimensions} can be
described as follows.  Suppose $(W,d\lambda)$ is a Liouville cobordism
from $(M,\xi)$ to a standard contact sphere $(S^{2n-1},\xi\std)$, and
let $(\overline{W},d\lambda)$ denote the completion obtained by attaching
cylindrical ends in the standard way (see \S\ref{sec:curves}).
Then the positive end of $\overline{W}$ can be assumed to match the
top half of the symplectization 
\begin{equation}
\label{eqn:symplectizationStandard}
\left(\RR \times S^{2n-1},d(e^r \alpha\std)\right),
\end{equation}
where $\alpha\std$ is the standard contact form, defined by restricting the 
Liouville form $\lambda\std := \sum_{j=1}^n (x_j \, dy_j - y_j \, dx_j)$
to the unit sphere.  We will assume also that the negative end matches
$\left((-\infty,0] \times M, d(e^r \alpha)\right)$ 
where $\alpha$ is (after a positive
rescaling) an arbitrary nondegenerate contact form for $(M,\xi)$.
(The nondegeneracy assumption was not included in 
Theorem~\ref{thm:allDimensions}, but this assumption will be easy to remove
in the final step, see Remark~\ref{remark:degenerate} below.)

The first step in the proof
is then to choose a suitable almost complex structure $J$ on the
symplectization \eqref{eqn:symplectizationStandard} that admits a foliation by a $(2n-2)$-dimensional family
of $J$-holomorphic planes, 
so-called ``seed curves,'' which are asymptotic to a fixed Reeb orbit $\gamma$
for $\alpha\std$ that has the smallest possible period.  We will be able
to verify explicitly that these planes are \emph{Fredholm regular}
for the moduli problem with fixed asymptotic orbit, hence the moduli
space is cut out transversely, and moreover, there exist no other curves
in $\RR \times S^{2n-1}$ with a single positive end approaching~$\gamma$.
Once these curves are understood, we can regard them as living in the
cylindrical end $[0,\infty) \times S^{2n-1} \subset \overline{W}$, so after
extending $J$ to a compatible almost
complex structure on the rest of $(\overline{W},d\lambda)$, they generate
a nonempty moduli space $\mM(J)$ of unparametrized $J$-holomorphic planes
in $\overline{W}$, all asymptotic to the same simply covered 
Reeb orbit in the sphere,
and this moduli space is a smooth $(2n-2)$-dimensional manifold for
generic extensions of $J$ since all curves in $\mM(J)$ are somewhere
injective.  Our main task is then to understand the natural compactification
$\overline{\mM}(J)$ of~$\mM(J)$, that is to say, the closure of $\mM(J)$
in the space of stable $J$-holomorphic buildings in the sense of \cite{SFTcompactness}.
Recall that a $J$-holomorphic building in a cobordism may have multiple
\emph{levels}, including one \defin{main level} which is a (possibly empty) curve in the completed
cobordism, and arbitrary finite numbers of \defin{upper levels} living
in the symplectization of the convex boundary and \defin{lower levels} living
in the symplectization of the concave boundary.
The uniqueness of the seed curves in the positive end implies the following:

\begin{lem}
\label{lemma:noUpperLevel}
If $u \in \overline{\mM}(J)$ is a stable holomorphic building with a nontrivial
upper level, then it has exactly one upper level, which consists of one of
the seed curves in $\RR \times S^{2n-1}$, and all its other levels are
empty.
\qed
\end{lem}

The lemma means that the only way for a sequence of planes in
$\mM(J)$ to ``degenerate'' with something nontrivial happening at the positive
end is if the planes simply escape into the positive end and become seed 
curves; in particular, this cannot happen to any sequence of planes that have
points falling into the negative end.  Theorem~\ref{thm:allDimensions} can
now be proved as follows.  Let $\mM_1(J)$ denote the smooth $2n$-dimensional
moduli space consisting of curves in $\mM(J)$ with the additional data of
a marked point, hence there is a well-defined evaluation map
$$
\ev : \mM_1(J) \to \overline{W}.
$$
Choose a smooth properly embedded $1$-dimensional submanifold
$\ell \subset \overline{W}$ with one end in $[0,\infty) \times S^{2n-1}$
and the other in $(-\infty,0] \times M$, and perturb it to be transverse
to the evaluation map.  Then
$$
\mM_\ell(J) := \ev^{-1}(\ell)
$$
is a smooth $1$-dimensional manifold, and it has a unique connected
component $\mM_\ell^0(J) \subset \mM_\ell(J)$ that contains seed curves
in the positive end.  This component has a noncompact end consisting of a
family of seed curves that escape to $+\infty$, thus it
is manifestly noncompact and therefore diffeomorphic to~$\RR$.
We claim now that $\mM_\ell^0(J)$ must also
contain curves with points that descend arbitrarily far into the
negative end.  Indeed, the SFT compactness theorem would otherwise
imply that every sequence in $\mM_\ell^0(J)$ has a subsequence convergent
to either an element of $\mM_\ell^0(J)$ or a holomorphic building of the
type described in Lemma~\ref{lemma:noUpperLevel}.  But the latter can only
happen if the sequence escapes through the neighbourhood of $+\infty$
in which all curves are seed curves.  In particular, we obtain a contradiction
by considering a noncompact sequence escaping to the \emph{opposite} end
of $\mM_\ell^0(J) \cong \RR$ from the one consisting of seed curves,
and this proves the claim.
It follows that one can find a sequence
$u_k \in \mM_\ell(J)$ of curves converging to a holomorphic building
$u_\infty \in \overline{\mM}(J)$ with a nontrivial lower level
(see Figure~\ref{fig:planes}).
Since the cobordism is exact, every component curve in $u_\infty$ must
have exactly one positive end, and it follows that at least one of the
curves in a lower level of $u_\infty$ is a plane, whose asymptotic orbit
is the contractible Reeb orbit promised by Theorem~\ref{thm:allDimensions}.

\begin{figure}
\includegraphics[width=5in]{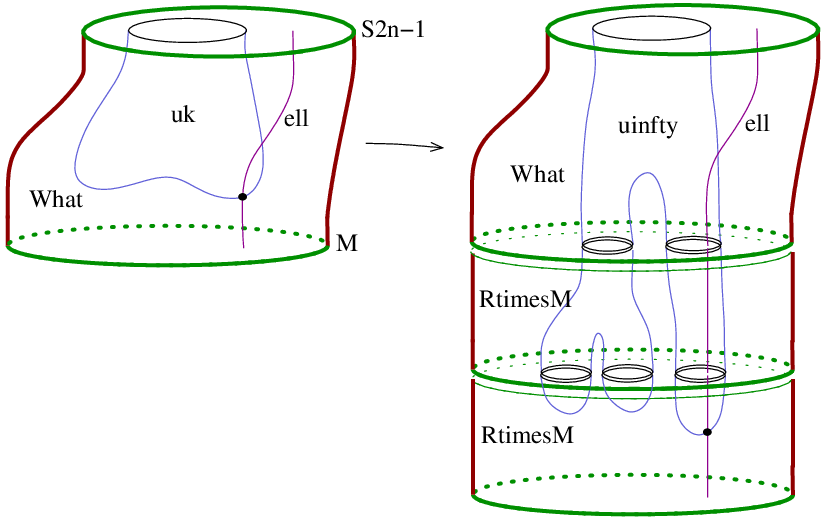}
\caption{\label{fig:planes} When holomorphic planes in an exact cobordism
converge to a holomorphic building with nontrivial lower levels,
at least one of them must include a plane.}
\end{figure}

\begin{remark}
\label{remark:degenerate}
To remove the nondegeneracy assumption from Theorem~\ref{thm:allDimensions},
one can take advantage of the fact that due to the exactness of the
cobordism, the contractible orbit found in the above argument comes with
an a priori bound on its period.  Then if $\alpha$ is a degenerate contact
form on $(M,\xi)$ approximated by a sequence $\alpha_k$ of nondegenerate
contact forms, the above argument gives a sequence $\gamma_k$ of contractible
Reeb orbits with respect to $\alpha_k$ whose periods are uniformly bounded,
so by Arzel\`a-Ascoli, these have a subsequence convergent to a contractible
Reeb orbit with respect to~$\alpha$.  Note that if the orbits $\gamma_k$
are also unknotted, it is not so clear whether the limiting orbit will also
be unknotted, hence the need for the more technical
Theorem~\ref{thm:technical}.
\end{remark}

\subsection{Outline of proofs, part~2: intersections}

The argument described thus far is quite standard and, as mentioned earlier,
is largely attributable to Hofer \cite{Hofer:weinstein}
(though the use of the path $\ell \subset \overline{W}$ to define a
$1$-dimensional submanifold of the moduli space is borrowed from
Niederkr\"uger \cite{Plastikstufe}).  The arguments required for finding
an orbit that is not only contractible but also \emph{unknotted} are
significantly subtler, and here we must make liberal use of
Siefring's intersection theory \cite{Siefring:intersection} in the
low-dimensional setting.

To explain the idea, we briefly recall the notion of \defin{nicely embedded} 
holomorphic curves, introduced by the second author in 
\cites{Wendl:compactnessRinvt,Wendl:automatic}.  The precise definition
will be reviewed in \S\ref{sec:intersection}, but in essence, a 
holomorphic curve $u : \dot{\Sigma} \to \overline{W}$ in a completed 
$4$-dimensional symplectic cobordism $\overline{W}$ is
nicely embedded if it has the necessary intersection-theoretic properties
to guarantee that it \emph{does not intersect its neighbors} in the moduli space.
This condition implies that the moduli space near~$u$ can be at most
$2$-dimensional, and in the $2$-dimensional case the
curves near $u$ form the leaves of a foliation on a neighbourhood of 
$u(\dot{\Sigma})$ in~$\overline{W}$.  If $\overline{W}$ is a symplectization
$\RR \times M$ or the image of $u$ is confined to a cylindrical end, then
being nicely embedded has the additional implication that $u$ projects
to an embedding into~$M$, i.e.~$u$ can be written as
$$
u = (u_\RR,u_M) : \dot{\Sigma} \to \RR \times M,
$$
where the map $u_M : \dot{\Sigma} \to M$ is also an embedding.
It is easy to show that the seed curves we find in the symplectization
of $(S^3,\xi\std)$ are nicely embedded, and the homotopy invariance of the
intersection theory then implies that the same is true for all
curves in~$\mM(J)$.

The fundamental principle behind
the proof of Theorems~\ref{thm:main} and~\ref{thm:nonprime} is then the
notion that ``nice curves degenerate nicely,'' i.e.~if a sequence
$u_k \in \mM(J)$ converges to a holomorphic building
$u_\infty \in \overline{\mM}(J)$, then we should expect the component
curves in levels of $u_\infty$ to be nicely embedded.  This statement
as such is false in full generality (see 
\cite{Wendl:automatic}*{Example~4.22 and Remark~4.23} for counterexamples),
but we will show that it is true in the present situation.
As a consequence, the plane we find in a lower level of $u_\infty$ has
the form $(u_\RR,u_M) : \CC \to \RR \times M$, where
$u_M : \CC \to M$ is an embedding asymptotic to a contractible Reeb orbit.

There remains one complication: the fact that $u : \CC \to \RR \times M$
is nicely embedded does not guarantee that its asymptotic orbit must be
simply covered, i.e.~the image of $u_M : \CC \to M$ might look like an
immersed disk that is embedded on the interior but multiply covered
on its boundary.  We will show in fact that this can happen, but only
in very specific ways, and to prove it, we develop a ``local adjunction
formula'' for holomorphic annuli breaking along a Reeb orbit.

\subsection{Local adjunction}

We now briefly interrupt the outline of the proof to describe a tool of
more general applicability.  To set the stage, suppose that 
$\alpha_k \to \alpha_\infty$ is a $\mc{C}^\infty$-convergent sequence of 
contact forms on a $3$-manifold~$M$, and $J_k \to J_\infty$ is a corresponding 
sequence with each $J_k$ belonging to the usual space (see \S\ref{sec:definitions})
of admissible translation-invariant almost complex structures on the symplectization 
$(\RR \times M,d(e^r \alpha_k))$.  Assume then that
$$
u_k : ([-k,k] \times S^1,i) \to (\RR \times M,J_k)
$$
is a sequence of pseudoholomorphic annuli which are converging in the
sense of SFT compactness to a broken $J_\infty$-holomorphic curve
$$
u_k \to ( u_\infty^+ | u_\infty^- ),
$$
where $u_\infty^+$ is the top level with a negative puncture, 
and $u_\infty^-$ is the bottom level with a
positive puncture, both asymptotic to the same nondegenerate Reeb orbit
$\gamma$ with covering multiplicity~$m(\gamma)$.  It is natural to choose
holomorphic cylindrical coordinates around these punctures and thus
parametrize the two levels in the form
\begin{equation*}
\begin{split}
u_\infty^+ : ((-\infty,0] \times S^1,i) &\to (\RR \times M,J_\infty), \\
u_\infty^- : ([0,\infty) \times S^1,i) &\to (\RR \times M,J_\infty),
\end{split}
\end{equation*}
so that the two half-cylinders together can be regarded as a broken
holomorphic annulus arising as a limit of the finite (but increasingly long)
holomorphic annuli~$u_k$; see Figure~\ref{fig:annulus}.  This is intended as
a local picture of the neighbourhood of a breaking orbit as a sequence
of smooth finite energy curves converges to a holomorphic building
as in \cite{SFTcompactness}.  

Recall from \cite{Siefring:asymptotics} that 
for any finite energy punctured holomorphic curve that is not a multiple
cover, sufficiently small neighbourhoods of each puncture are always
embedded, hence if $u_\infty^+$ and $u_\infty^-$ are not multiply covered
then we are free to assume without loss of generality that both are
embedded.  This implies that each $u_k$ is also embedded near the boundary
of $[-k,k] \times S^1$ for sufficiently large~$k$, but if $m(\gamma) > 1$,
then $u_k$ can have finitely many double points and critical points that 
``disappear into the breaking orbit'' in the limit.
See  \S\ref{sec:intersection} for precise definitions of each of the quantities
discussed below.  We let
$$
\delta(u_k) \ge 0
$$
denote the algebraic count of double points and critical points
of~$u_k$: this is a nonnegative integer that equals zero if and only if
$u_k$ is embedded.  The half-cylinders $u_\infty^\pm$ are embedded by
assumption, but if $m(\gamma) > 1$, then they may have 
``hidden double points at infinity'' in
the sense of \cite{Siefring:intersection}, i.e.~double points that must
emerge from infinity under generic perturbations of the curves.
We denote the algebraic counts of these hidden double points by
$$
\delta_\infty(u_\infty^\pm) \ge 0;
$$
they are nonnegative integers that vanish if and only if generic perturbations
of $u_\infty^\pm$ remain embedded.  We denote by
$$
\bar{\sigma}_\pm(\gamma) \ge 1
$$
the so-called \defin{spectral covering numbers} of~$\gamma$ as in
\cite{Siefring:intersection}: these are covering multiplicities
of certain asymptotic eigenfunctions of $\gamma$, and are thus positive
integers that equal~$1$ if and only if those eigenfunctions are simply
covered (which is always the case e.g.~if $m(\gamma)=1$).  
For one last piece of notation, we let
$$
p(\gamma) \in \{0,1\}
$$
denote the \defin{parity} of~$\gamma$, i.e.~its Conley-Zehnder index modulo~$2$.
The result we will prove in \S\ref{sec:localAdjunction} 
can now be stated as follows.

\begin{figure}
\includegraphics[width=4in]{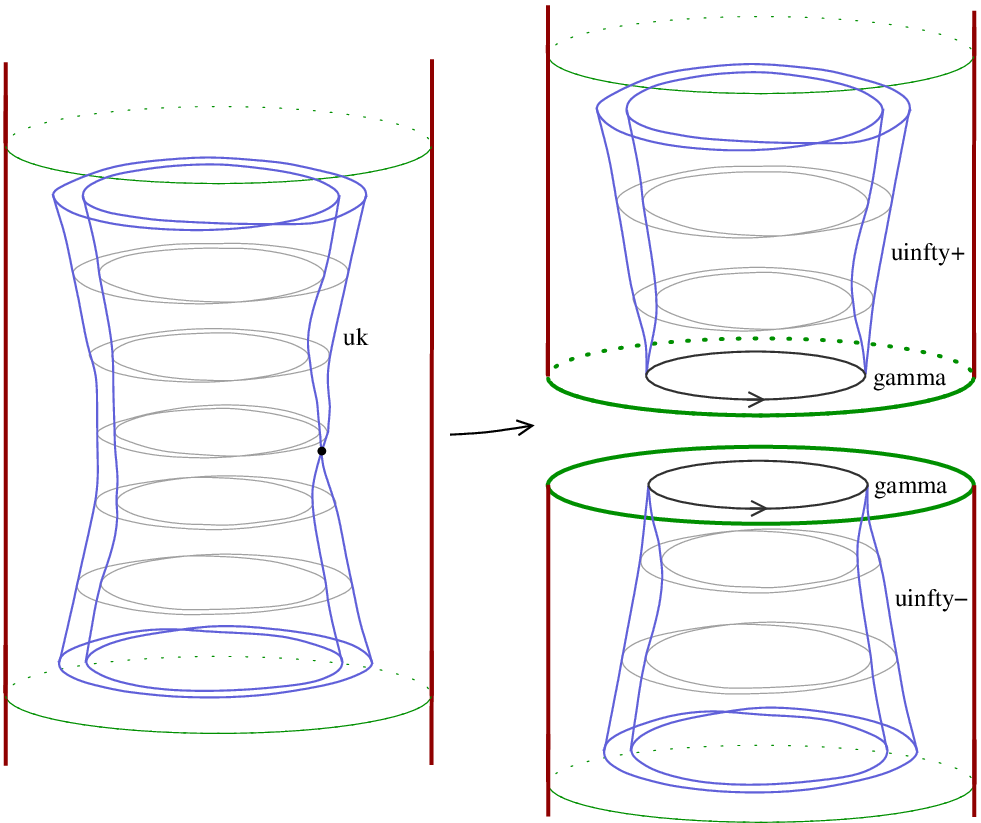}
\caption{\label{fig:annulus} A sequence of pseudoholomorphic annuli $u_k$
converging to a broken annulus consisting of embedded half-cylinders
$u_\infty^\pm$ asymptotic to a doubly covered breaking orbit~$\gamma$.
In this case, $u_k$ can have double points that disappear in the limit.}
\end{figure}

\begin{thm}[local adjunction]
\label{thm:localAdjunction}
In the setting described above, 
assume $u_k \to (u_\infty^+ | u_\infty^-)$ 
is a sequence of holomorphic annuli in $\RR \times M$ converging to a
broken pair of half-cylinders, where $u_\infty^+$ and
$u_\infty^-$ are both embedded and asymptotic to a nondegenerate
Reeb orbit $\gamma$
with covering multiplicity $m(\gamma)$, parity $p(\gamma)$ and spectral
covering numbers $\bar{\sigma}_\pm(\gamma)$.  Then for all $k$ sufficiently
large,
\begin{multline*}
2\delta(u_k) = 2[ \delta_\infty(u_\infty^+) + \delta_\infty(u_\infty^-) ]
+ \left[ \bar{\sigma}_+(\gamma) - 1 \right] \\
 \quad + \left[ \bar{\sigma}_-(\gamma) - 1 \right]
+ \left[m(\gamma) - 1 \right] p(\gamma).
\end{multline*}
\end{thm}

The usefulness of this theorem lies in the fact that every bracketed
term on the right hand side of the formula is known
a priori to be nonnegative, so if we also know that the annuli $u_k$ are
embedded, then all these terms must vanish.  In that case, we will easily
be able to deduce the following consequence:

\begin{cor}
\label{cor:localAdjunction}
In the setting of Theorem~\ref{thm:localAdjunction}, if $u_k$ is embedded
for every~$k$, then one of the following is true:
\begin{itemize}
\item $\gamma$ is a simply covered orbit;
\item $\gamma$ is a double cover of a simply covered orbit $\gamma'$
such that $p(\gamma')=1$ and $p(\gamma)=0$, and both of the
half-cylinders $u_\infty^\pm$ have no hidden double points at infinity.
\end{itemize}
\end{cor}

\subsection{Outline of proofs, conclusion}

In the situation at hand, our degenerating curves are all embedded, so
Corollary~\ref{cor:localAdjunction} applies
and we conclude that the breaking orbit is always
either simply covered or a double cover of a negative hyperbolic
orbit, what is known in the SFT literature (cf.~\cite{SFT}) as a 
\defin{bad orbit}.
In the first case we are done, and in the second, we will show that
degenerations of this form can always be glued back together so that they
are interpreted as \emph{interior points} of the compactified moduli
space, and the moduli space must therefore have additional degenerations
besides this.  In other words, breaking along bad orbits can happen, but
it cannot be the only type of breaking that happens, so there is still
guaranteed to be some breaking along a simple orbit somewhere, producing
a nicely embedded curve asymptotic to an unknotted orbit.  The resulting
constraints on the Conley-Zehnder index and self-linking number of the
orbit then follow by a straightforward and essentially standard
topological computation.

The major differences between the above summary and the proof of
Theorem~\ref{thm:nonprime} are as follows.  For the first statement in the
theorem, the symplectic cobordism $W$ is taken to be symplectically trivial, 
i.e.~its completion has the
form $(\RR \times M,d\lambda)$, where $\lambda$ is a Liouville form
matching $e^r \alpha_\pm$ near $\{\pm\infty\} \times M$, and
$\alpha_\pm$ are two nondegenerate contact forms for $(M,\xi)$, of which
$\alpha_-$ is given but $\alpha_+$ is carefully chosen.
The assumptions of the theorem then allow us to choose $\alpha_+$
and a compatible almost complex structure $J_+$ near $+\infty$
so that we find a smooth $1$-dimensional moduli space of seed curves.
Since this moduli space is only $1$- and not $2$-dimensional, it does not
form a foliation, but the curves are still nicely embedded and the same
principles therefore apply: a variation on the same argument described
above leads to a nicely embedded plane asymptotic to a simple
Reeb orbit for~$\alpha_-$.

Here is an outline of the remainder of the paper.
In \S\ref{sec:preparation}, we clarify the essential definitions and
review the necessary facts about punctured holomorphic curves and their
intersection theory in dimension four.  The purpose of
\S\ref{sec:seed} is then to specify the data at the positive ends of our
symplectic cobordisms, construct the seed curves and prove that they are
Fredholm regular and nicely embedded.
Theorem~\ref{thm:localAdjunction} and Corollary~\ref{cor:localAdjunction}
on local adjunction for breaking holomorphic annuli are proved
in \S\ref{sec:localAdjunction}.  Finally, \S\ref{sec:compactness}
carries out the main compactness arguments, and \S\ref{sec:proofs} 
completes the proofs of the main theorems.

\subsection{Acknowledgements}
The second author would like to thank Paolo Ghiggini for enlightening
discussions, and Emmy Murphy for explaining
Example~\ref{ex:cap} and pointing out the connection between
exact Lagrangian caps and Liouville cobordisms.
The present paper constitutes a portion of the first author's PhD thesis
within the framework of the project 
\textsl{Intersections in Low-Dimensional Symplectic Field Theory}, 
funded by the Leverhulme Trust.  The work of the second author was also 
partially funded by a Royal Society University Research Fellowship
and by EPSRC grant EP/K011588/1.

\section{Preparation}
\label{sec:preparation}

The purpose of this section is to fix definitions and review some known
results that will be needed in the rest of the paper.

\subsection{Contact manifolds and symplectic cobordisms}
\label{sec:definitions}

We begin by reviewing some basic definitions from contact geometry and the
precise way in which contact manifolds arise
as hypersurfaces or boundary components of symplectic manifolds.

Suppose $(W,\omega)$ is a $2n$-dimensional
symplectic manifold, and $M \subset W$ is a smooth oriented hypersurface.  We 
say that $M$ is \defin{convex} if there exists a Liouville vector
field near~$M$ that is positively transverse to~$M$: here a vector field $V$
is called \defin{Liouville} if its flow dilates the symplectic form,
meaning $\Lie_V \omega = \omega$.  This is equivalent
to the condition that the dual $1$-form $\lambda := \omega(V,\cdot)$
satisfies $d\lambda = \omega$, and being positively transverse to $M$ then
means that the restriction $\alpha := \lambda|_{TM}$ satisfies
$$
\alpha \wedge (d\alpha)^{n-1} > 0.
$$
This makes $\alpha$ a (positive) \defin{contact form} on~$M$, and the induced
(positive and co-oriented) contact structure is the co-oriented hyperplane 
field $\xi := \ker \alpha \subset TM$.  It follows from Gray's stability
theorem that if $V$ is replaced with any
other Liouville vector field positively transverse to~$M$, then the induced
contact structure is isotopic to~$\xi$, hence the contact form can be
regarded as an auxiliary choice, but the contact structure is canonical
up to isotopy.  

\begin{remark}
In this paper, every contact structure is assumed to be
\emph{co-oriented} and \emph{positive} (with respect to a given 
orientation of the manifold),
and contact forms are always assumed compatible with the given co-orientation.
\end{remark}

\begin{exmp}
\label{ex:stdSphere}
We denote by
$\xi\std \subset T S^{2n-1}$ the \defin{standard contact structure} on the
sphere, which arises as the convex boundary of the standard symplectic
unit ball with a Liouville vector field pointing radially outward.
In coordinates $(x_1,y_1,\ldots,x_n,y_n) \in \RR^{2n}$, the 
\defin{standard contact form} $\alpha\std$ is the restriction to 
$S^{2n-1} \subset \RR^{2n}$ of the Liouville form $\frac{1}{2} 
\sum_{j=1}^n (x_j\, dy_j - y_j\, dx_j)$.
\end{exmp}

Any choice of contact form $\alpha$ determines a
\defin{Reeb vector field} $R_\alpha$ on $M$ via the conditions
$$
d\alpha(R_\alpha,\cdot) \equiv 0, \qquad \alpha(R_\alpha) \equiv 1.
$$
If $M$ is a convex hypersurface in a symplectic manifold $(W,\omega)$, then
the orbits of $R_\alpha$ are precisely the orbits on $M$ of any Hamiltonian
vector field defined by a Hamiltonian function on $(W,\omega)$ with $M$ as a 
regular level set; moreover, convexity implies that a neighbourhood of $M$ is
foliated by other convex hypersurfaces that have the same Reeb orbits
up to parametrization.
See \cite{Geiges:book} for more on contact structures, and
\cite{HoferZehnder} for more on the convexity condition in Hamiltonian
dynamics.

Given two closed contact manifolds
$(M_-,\xi_-)$ and $(M_+,\xi_+)$, a \defin{strong symplectic cobordism}
from $(M_-,\xi_-)$ to $(M_+,\xi_+)$ is a compact symplectic manifold
$(W,\omega)$ whose boundary can be identified with $-M_- \sqcup M_+$
such that $M_-$ and $M_+$ are both convex hypersurfaces and the
contact structures they inherit are isotopic to $\xi_-$ and $\xi_+$ 
respectively.
Note that the orientation reversal for $M_-$ means that the Liouville
vector field points \emph{inward} at~$M_-$ (for this reason we sometimes
call $M_-$ the \defin{concave} boundary component), whereas it points
outward at~$M_+$.  Additionally, $(W,\omega)$ is called a \defin{Liouville}
(or \defin{exact symplectic}) cobordism from $(M_-,\xi_-)$ to $(M_+,\xi_+)$
if the transverse Liouville vector field defined near $\p W$ can be assumed
to extend to a global Liouville vector field.  This is equivalent to 
requiring $\omega = d\lambda$ for some $1$-form $\lambda$ that restricts to
the boundary as contact forms $\alpha_\pm := \lambda|_{TM_\pm}$ 
for~$\xi_\pm$.

The \defin{symplectization} of a contact manifold $(M,\xi = \ker\alpha)$ is the
open symplectic manifold $(\RR \times M,d(e^r \alpha))$, where $r$ denotes
the coordinate on~$\RR$.  Its symplectic
structure is independent of the choice of $\alpha$ up to isotopy, but
$\alpha$ determines a special class of compatible almost complex structures
$\jJ(\alpha)$
on $(\RR \times M,d(e^r\alpha))$ such that $J \in \jJ(\alpha)$ if and only if:
\begin{itemize}
\item $J$ is $\RR$-invariant (i.e.~invariant under the flow of~$\p_r$);
\item $J \p_r = R_\alpha$;
\item $J(\xi) = \xi$;
\item $d\alpha(\cdot,J\cdot)|_\xi$ is a bundle metric on~$\xi$.
\end{itemize}
Given a symplectic cobordism $(W,\omega)$ from $(M_-,\xi_-)$ to $(M_+,\xi_+)$
with induced contact forms $\alpha_\pm$ at~$M_\pm$, the corresponding
Liouville vector fields defined near $M_+$ and $M_-$ determine collar 
neighbourhoods $(-\epsilon,0] \times M_+$ and $[0,\epsilon) \times M_-$
respectively in which $\omega = d(e^r \alpha_\pm)$.  One then defines the
\defin{symplectic completion}
$$
\overline{W} = \left( (-\infty,0] \times M_- \right) \cup_{M_-} W
\cup_{M_+} \left( [0,\infty) \times M_+ \right)
$$
by extending $\omega$ over the cylindrical ends as $d(e^r \alpha_\pm)$.
We shall denote by
$$
\jJ(W,\omega,\alpha_+,\alpha_-)
$$
the (nonempty and contractible) space of almost complex structures
on $\overline{W}$ that are $\omega$-compatible on~$W$ and restrict to the
cylindrical ends as elements of~$\jJ(\alpha_\pm)$.  Almost complex structures 
of this type will be referred to simply as \defin{admissible} whenever the
corresponding symplectic and contact data is fixed.

\subsection{Reeb orbits and the Conley-Zehnder index}
\label{sec:Reeb}

Given a contact form $\alpha$ on a contact manifold
$(M,\xi)$ of dimension $2n-1$, a closed Reeb orbit can be regarded as a 
smooth map
$$
\gamma : S^1 := \RR / \ZZ \to M
$$
satisfying $\dot{\gamma} = T R_\alpha(\gamma)$ for some $T > 0$, which is
the orbit's \defin{period}.  Indeed, setting $x(t) := \gamma(t/T)$, such a 
map is equivalent to a path $x : \RR \to M$ that satisfies
$\dot{x} = R_\alpha(x)$ and $x(t+T) = x(t)$ for all~$t$.  The number $T$
need not generally be the \defin{minimal period}, hence $\gamma$ may
be a multiple cover $\gamma(t) = \gamma_0(kt)$ of another closed Reeb
orbit $\gamma_0$ for some integer $k \ge 2$; when this is not the case,
we say $\gamma$ is \defin{simple}, and the map $\gamma : S^1 \to M$ is
then an embedding.  When $\gamma$ is simple and $\dim M = 3$, it makes
sense to ask whether $\gamma$ is \defin{unknotted}, meaning it is
the boundary of an embedded disk, or more explicitly there exists 
an embedding
$$
u: \mb{D}^2 \hookrightarrow M
$$
whose restriction to the boundary coincides with the Reeb orbit:
$$
u|_{\p \mb{D}^2}=\c .
$$

To every closed Reeb orbit one can associate an integer-valued invariant, 
the \emph{Conley-Zehnder index},
which depends on a trivialization of the contact structure along the orbit.
We will recall the definition of this invariant by way of a 
theorem regarding \emph{asymptotic operators}.

Fix $J \in \jJ(\alpha)$ and
suppose $\c : S^1 \to M$ is a closed orbit of $R_\alpha$ with period~$T$.
Given any symmetric connection $\nabla$ on~$M$, define 
$A_\c:\mc{C}^\i(\c^*\xi)\rightarrow \mc{C}^\i(\c^*\xi)$ by
\begin{equation}
A_\c \eta = -J(\nabla_t \eta-T\nabla_\eta R_\alpha).
\end{equation}
This operator is well defined and independent of the choice of connection~$\nabla$
(see e.g.~\cite{Wendl:SFT}*{\S 3.3}), and it is
symmetric with respect to the inner product on $\mc{C}^\i(\c^*\xi)$ defined by
$$
\langle \eta, \zeta \rangle = \int_{S^1} \o_{\c(t)}\big(\eta(t), J(\c(t))\zeta(t)\big) \, dt.
$$
It also extends to an unbounded self-adjoint operator on
$L^2(\c^*\xi)$ with domain $W^{1,2}(\c^*\xi)$, referred to
as the \defin{asymptotic operator} associated to~$\c$.
Its spectral properties have been described in \cite{HWZ:props2}.
\begin{prop}[\cite{HWZ:props2}]
\label{spectral}
With the notation above, let $\sigma(A_\c) \subset \RR$ denote the 
spectrum of $A_\c$, and for any
$\lambda \in \sigma(A_\c)$, denote the corresponding eigenspace by~$E_\lambda$.
Then:
\begin{enumerate}
\item $0 \in \sigma(A_\c)$ if and only if $\c$ is degenerate;
\item $\sigma(A_\c)$ is a discrete subset;
\item For each $\lambda \in \sigma(A_\c)$, $1 \le \dim E_\lambda \le 2(n-1)$;
\item All nontrivial eigenfunctions of $A_\c$ are everywhere nonzero.
\end{enumerate}
If $\dim M = 3$, then the last statement implies that one can define
winding numbers $\wind^\Phi(\eta) \in \ZZ$ of nontrivial eigenfunctions
$\eta$ relative to 
any fixed unitary trivialization $\Phi$ of $\c^*\xi$.
The following statements then also hold:
\begin{enumerate}
\setcounter{enumi}{4}
\item If $\eta, \zeta \in E_\lambda$ are two nontrivial elements of the same
eigenspace, then $\wind^\Phi(\eta) = \wind^\Phi(\zeta)$,
hence we can sensibly denote both by~$\wind^\Phi(\lambda)$.
\item The map $\sigma(A_\c) \to \ZZ : \lambda \mapsto \wind^\Phi(\lambda)$ 
is $2$-to-$1$ (counting
multiplicity of eigenvalues) and increasing.  Hence if two distinct 
eigenvalues have the same winding, they are consecutive and their eigenspaces 
are $1$-dimensional.
\end{enumerate}
\end{prop}

It follows that one can speak of the largest 
negative eigenvalue and the smallest positive eigenvalue associated to 
the asymptotic operator, and when $\dim M = 3$, their winding numbers relative
to a chosen trivialization $\Phi$ are denoted by
$$
\a_-^\Phi(\c) , \quad \a_+^\Phi(\c) \in \ZZ
$$
respectively.  Proposition~\ref{spectral} implies that these two numbers differ
by either $0$ or $1$ if $\c$ is nondegenerate, and in this case,
the \defin{Conley-Zehnder index} relative to the 
trivialization $\Phi$ of $\c^*\xi$ can be characterized 
(according to a theorem in \cite{HWZ:props2}) via the relation
\begin{equation}
\label{eqn:CZwinding}
\cz{\c}{\Phi} =\a_-^\Phi(\c) + \a_+^\Phi(\c) \in \ZZ,
\end{equation}
and its \defin{parity} (which does not depend on~$\Phi$) by 
\begin{equation}
\label{eqn:parity}
p(\c)=\a_+^\Phi(\c) -\a_-^\Phi(\c) \in \{0,1\}.
\end{equation}
As these formulas indicate, $\cz{\c}{\Phi}$ depends only on the asymptotic
operator and can thus sensibly be written as 
$$
\cz{A_\c}{\Phi} = \cz{\c}{\Phi}.
$$
With this in mind, \eqref{eqn:CZwinding} can also be used to compute
Conley-Zehnder indices in higher dimensions, via the relation
\begin{equation}
\label{eqn:CZhigherDim}
\cz{A_1 \oplus \ldots \oplus A_m}{\Phi_1 \oplus \ldots \oplus \Phi_m} =
\cz{A_1}{\Phi_1} + \ldots + \cz{A_m}{\Phi_m},
\end{equation}
which holds for any collection of asymptotic operators $A_j$ with trivial
kernels on Hermitian line bundles trivialized by $\Phi_j$ for $j=1,\ldots,m$.
We will use this to compute the indices of higher-dimensional
seed curves in \S\ref{sec:stdSphere}.

While $\cz{\c}{\Phi}$ depends generally on the choice of trivialization~$\Phi$,
in certain situations one can make natural choices to remove this ambiguity.
If $\gamma$ is nullhomologous and forms the boundary of an immersed 
surface $\dD$ in~$M$, we define
$$
\muCZ(\gamma ; \dD) \in \ZZ
$$
as $\cz{\c}{\Phi}$ with $\Phi$ required to admit an extension to a unitary
trivialization of $\xi$ along~$\dD$.  The index in this case still depends on 
the choice of surface~$\dD$, but this ambiguity also disappears if $c_1(\xi) = 0$,
which is true e.g.~on $(S^3,\xi\std)$.

We require the following standard lemma on the behaviour of the index for
multiply covered orbits in dimension three.  Let
$$
\gamma^k : S^1 \to M : t \mapsto \gamma(kt)
$$
denote the $k$-fold cover of the orbit $\gamma : S^1 \to M$ for $k \in \NN$,
and note that any trivialization $\Phi$ of $\gamma^*\xi$ induces a
trivialization $\Phi^k$ of $(\gamma^k)^*\xi$.
\begin{lem} \label{czcover}
Suppose $\dim M = 3$, and that $\c$ and all its multiple covers are 
nondegenerate.  Then for any unitary trivialization $\Phi$ of $\c^*\xi$,
\begin{equation}
\cz{\c^k}{\Phi^k} =
\begin{cases}
k\.\cz{\c}{\Phi} & \text{ if $\c$ is hyperbolic}\\
2 \lfloor{k \theta}\rfloor+1 & \text{ if $\c$ is elliptic}
\end{cases}
\end{equation}
for every $k \in \NN$, where in the elliptic case,
$\theta \in \RR$ is an irrational number determined by $\gamma$ and~$\Phi$.
\end{lem}

We will occasionally also need to deal with Reeb orbits $\c$ that are degenerate but
belong to Morse-Bott families, in which case the following definition will be
convenient.  If $\c$ is degenerate, then $0 \in \sigma(A_\c)$ but one can
find $\epsilon > 0$ such that $(-\epsilon,0) \cap \sigma(A_\c) = \emptyset$.
It follows that for any $\epsilon > 0$ sufficiently small, $A_\c + \epsilon$
is the asymptotic operator of a perturbed nondegenerate orbit, whose index
we will denote by
\begin{equation}
\label{eqn:perturbedCZ}
\cz{\c+\epsilon}{\Phi} := \cz{A_\c + \epsilon}{\Phi}.
\end{equation}
This is independent of the choice as long as $\epsilon > 0$ is sufficiently
small, and this \defin{perturbed Conley-Zehnder index} gives a sharp lower 
bound on the indices of possible nondegenerate perturbations of~$\c$.
The winding numbers $\a_\pm^\Phi(\c + \epsilon) \in \ZZ$ are defined
similarly after replacing $A_\c$ by $A_\c + \epsilon$, and they are then
related to $\cz{\c+\epsilon}{\Phi}$ by the obvious analogue of
\eqref{eqn:CZwinding}.  Notice that $\a_-^\Phi(\c + \epsilon) = \a_-^\Phi(\c)$,
but $\a_+^\Phi(\c + \epsilon)$ and $\a_+^\Phi(\c)$ may differ if $\c$
is degenerate.

Finally, here is a definition that will be needed for intersection
theory when $\dim M = 3$.  Observe that for any Reeb orbit $\gamma_0$
and integers $k \ge 2$, every 
eigenfunction in the $\lambda$-eigenspace of $A_{\gamma_0}$ has 
a $k$-fold cover 
that belongs to the $k\lambda$-eigenspace of $A_{\gamma_0^k}$.  In the
three-dimensional case, one can use Proposition~\ref{spectral} to show that
the covering multiplicity of an eigenfunction depends only on its winding
number, thus all elements of the same eigenspace have the same covering
multiplicity.  The (positive and negative) \defin{spectral covering numbers}
$$
\bar{\sigma}_\pm(\gamma) \in \NN
$$
are defined as the covering multiplicity of the eigenspace that has 
winding~$\a_\pm^\Phi(\gamma)$.  Note that this is only 
interesting when $\gamma = \gamma_0^k$ for some other orbit $\gamma_0$
and $k \ge 2$; if $\gamma$ is
simple then $\bar{\sigma}_\pm(\gamma) = 1$ always.

\subsection{Holomorphic curves in completed symplectic cobordisms}
\label{sec:curves}

In this subsection, 
fix a $2n$-dimensional symplectic cobordism $(W,\omega)$ with completion 
$\overline{W}$ and
admissible almost complex structure $J \in \jJ(W,\omega,\alpha_+,\alpha_-)$,
with the restrictions of $J$ to the cylindrical ends denoted by
$J_\pm \in \jJ(\alpha_\pm)$.

\subsubsection{Asymptotics}

We will consider asymptotically cylindrical 
pseudoholomorphic curves $u : (\dot{\Sigma},j) \to (\overline{W},J)$,
where 
$$
\dot{\Sigma} = \Sigma \setminus \Gamma
$$
is the result of removing finitely many punctures $\Gamma \subset \Sigma$
from a closed Riemann surface~$(\Sigma,j)$.  The set of punctures is partitioned
into sets of \defin{positive} and \defin{negative} punctures 
$\Gamma^+$ and $\Gamma^-$ respectively, where $z \in \Gamma^\pm$ means that
one can find a biholomorphic
identification of a punctured neighbourhood of $z$ with 
$[0,\infty) \times S^1$ or $(-\infty,0] \times S^1$ respectively such that
for $|s|$ sufficiently large, $u$ in these coordinates takes the form
$$
u(s,t) = \exp_{(T s,\c(t))} h(s,t) \in [0,\infty) \times M_+ \text{ or }
(-\infty,0] \times M_-
$$
for some closed Reeb orbit $\c : S^1 \to M_\pm$ with period~$T > 0$, where
the exponential map is defined with respect to any choice of 
translation-invariant metric on the cylindrical ends, and
$h(s,t)$ is a vector field along the trivial cylinder which satisfies
$|h(s,t)| \to 0$ as $s \to \pm\infty$.  We say in this case that $u$ is 
(positively or negatively) \defin{asymptotic to~$\c$} at~$z$, and $h(s,t)$
is called the \defin{asymptotic representative} of $u$ at~$z$.  The asymptotic
behaviour of $h(s,t)$ is described by a formula proved in 
\cites{HWZ:props1,HWZ:props4,Mora,Siefring:asymptotics}: namely if
the orbit $\gamma$ is nondegenerate or Morse-Bott, then for
$|s|$ sufficiently large, $h$ is either identically zero or satisfies
\begin{equation}
\label{eqn:asympFormula}
h(s,t) = e^{\lambda s}(e_1(t)+r(s,t)),
\end{equation}
where $r(s,t) \rightarrow 0$ uniformly in all derivatives as 
$s \rightarrow \pm \infty$, $\lambda \in \sigma(A_\c)$ is an eigenvalue of
the asymptotic operator of $\c$ with $\pm\lambda < 0$, and $e_1 \in 
\mc{C}^\i(\c^*\xi_\pm)$ is a nontrivial element of the corresponding eigenspace.

\subsubsection{Moduli spaces and compactness}

It is a standard fact that every asymptotically cylindrical $J$-holomorphic
curve $u : (\dot{\Sigma},j) \to (\overline{W},J)$ either is somewhere injective
or is a multiple cover of a somewhere injective asymptotically cylindrical 
curve, and moreover, the
set of \emph{injective points} of a somewhere injective curve is open and
dense.  A complete proof of this statement may be found in \cite{Nelson:Abendblatt},
using asymptotic results of Siefring \cite{Siefring:asymptotics}.
Recall that $z \in \dot{\Sigma}$ is called an \defin{injective point} of $u$
if $u^{-1}(u(z)) = \{z\}$ and $du(z) \ne 0$, and we call $u$ a
$k$-fold \defin{multiple cover} of another curve $v : (\dot{\Sigma}' =
\Sigma' \setminus \Gamma',j') \to (\overline{W},J)$ if
$$
u = v \circ \phi
$$
for some holomorphic map $\phi : (\Sigma,j) \to (\Sigma',j')$ of degree~$k$.

Fix finite ordered tuples of Reeb orbits $\boldsymbol{\gamma}^+ = (\gamma_1^+,\ldots,
\gamma_{k_+}^+)$ and $\boldsymbol{\gamma}^- = (\gamma_1^-,\ldots,
\gamma_{k_-}^-)$ in $M_+$ and $M_-$ respectively (the case $k_\pm = 0$ is allowed),
assuming that all of them are either nondegenerate or belong to
Morse-Bott families.  For an integer $m \ge 0$, the moduli space
$$
\mM_m(J,\boldsymbol{\gamma}^+,\boldsymbol{\gamma}^-)
$$
of \defin{unparametrized $J$-holomorphic spheres
asymptotic to $\boldsymbol{\gamma}^+$ and $\boldsymbol{\gamma}^-$ with
$m$ marked points}
is defined as the set of equivalence classes of tuples
$(\Sigma,j,\Gamma^+,\Gamma^-,u,(\zeta_1,\ldots,\zeta_m))$ where 
$(\Sigma,j)$ is a closed Riemann surface
of genus zero, $\Gamma^+, \Gamma^- \subset \Sigma$ are disjoint finite sets, each
equipped with an ordering, the \defin{marked points} 
$\zeta_1,\ldots,\zeta_m \in \dot{\Sigma} :=
\Sigma \setminus (\Gamma^+ \cup \Gamma^-)$ are all distinct, and
$$
u : (\dot{\Sigma},j) \to (\overline{W},J)
$$
is an asymptotically cylindrical $J$-holomorphic curve
with positive punctures
$\Gamma^+$ and negative punctures $\Gamma^-$, such that $u$ is asymptotic at
the $i$th puncture in $\Gamma^\pm$ to $\gamma_i^\pm$ for $i=1,\ldots,k_\pm$.
Two such tuples are considered equivalent if one can be written as a
reparametrization of the other via a biholomorphic diffeomorphism of their
domains that maps marked points to marked points and punctures to punctures, 
with signs and orderings preserved.
The topology of $\mM_m(J,\boldsymbol{\gamma}^+,\boldsymbol{\gamma}^-)$ 
can be characterized by saying that a sequence converges if it has 
representatives with a fixed domain $\Sigma$ and fixed sets of punctures 
and marked points such that the conformal structures converge in 
$\mc{C}^\infty(\Sigma)$ while the maps to $\overline{W}$ converge in 
$\mc{C}^\infty_{\text{loc}}(\dot{\Sigma})$ and also in $\mc{C}^0$ up to 
infinity (with respect to translation-invariant metrics on the cylindrical
ends).
We shall often abuse notation by referring to the entire equivalence class
of tuples $[(\Sigma,j,\Gamma^+,\Gamma^-,u,(\zeta_1,\ldots,\zeta_m))]$ forming an element of
$\mM_m(J,\boldsymbol{\gamma}^+,\boldsymbol{\gamma}^-)$ simply as~$u$.
In this paper we will only consider the cases $m=0,1$, abbreviating the former
by
$$
\mM(J,\boldsymbol{\gamma}^+,\boldsymbol{\gamma}^-) :=
\mM_0(J,\boldsymbol{\gamma}^+,\boldsymbol{\gamma}^-).
$$
For $m>0$, the \defin{evaluation map}
\begin{equation*}
\begin{split}
\ev : \mM_m(J,\boldsymbol{\gamma}^+,\boldsymbol{\gamma}^-) &\to \overline{W}^m \\
[(\Sigma,j,\Gamma^+,\Gamma^-,u,(\zeta_1,\ldots,\zeta_m))] &\mapsto
(u(\zeta_1),\ldots,u(\zeta_m))
\end{split}
\end{equation*}
is well defined and continuous by construction.

Recall that neighbourhoods in 
$\mM_m(J,\boldsymbol{\gamma}^+,\boldsymbol{\gamma}^-)$ can be described
as zero-sets of smooth Fredholm sections in suitable Banach space bundles
(see e.g.~\cite{Wendl:automatic}).  A curve $u$ is called \defin{Fredholm
regular} whenever it forms a transverse intersection of such a Fredholm
section with the zero-section.  The \defin{virtual dimension} of
$\mM_m(J,\boldsymbol{\gamma}^+,\boldsymbol{\gamma}^-)$ at $u$ is given by
the Fredholm index of the linearized section at~$u$ minus the dimension of
the group of automorphisms of the domain, and in the case $m=0$ is also called 
the \defin{index} of~$u$.  If the orbits are all nondegenerate, it is given by
the formula
\begin{equation}
\label{eqn:index}
\ind{u} = (n-3) \chi(\dot{\Sigma}) + 2 c_1^\Phi(u^*T\overline{W}) +
\sum_{i=1}^{k_+} \cz{\gamma_i^+}{\Phi} - \sum_{i=1}^{k_-} \cz{\gamma_i^-}{\Phi}.
\end{equation}
Here $\Phi$ is an arbitrary choice of unitary trivializations of $\xi_\pm$ 
along each of the asymptotic orbits, which naturally induce asymptotic 
trivializations of the complex vector bundle
$u^*T\overline{W} \to \dot{\Sigma}$, and $c_1^\Phi(u^*T\overline{W}) \in \ZZ$ 
then denotes the \defin{relative first Chern number} of $u^*T\overline{W}$ 
with respect to these asymptotic
trivializations.  This term ensures that the total expression is independent of the
choice~$\Phi$.  We will also need a special case of the index formula under
Morse-Bott assumptions: if all positive asymptotic orbits are Morse-Bott (but
possibly degenerate) and all negative orbits are nondegenerate, then
\begin{equation}
\label{eqn:indexMB}
\ind{u} = (n-3) \chi(\dot{\Sigma}) + 2 c_1^\Phi(u^*T\overline{W}) +
\sum_{i=1}^{k_+} \cz{\gamma_i^+ + \epsilon}{\Phi} -
\sum_{i=1}^{k_-} \cz{\gamma_i^-},
\end{equation}
where $\epsilon > 0$ is assumed sufficiently small (see \eqref{eqn:perturbedCZ}).
Note that this is the virtual dimension of the moduli space of curves near $u$
with \emph{fixed} asymptotic orbits, i.e.~the orbits are not allowed to
move continuously in their respective Morse-Bott families.  The index without
this constraint would be larger; see \cite{Wendl:automatic} for an explanation
of \eqref{eqn:indexMB} and the constrained/unconstrained distinction.
Adding a marked point generally increases the virtual dimension by~$2$, so
$\mM_m(J,\boldsymbol{\gamma}^+,\boldsymbol{\gamma}^-)$ has virtual dimension
$\ind{u} + 2m$ on any component that includes the curve 
$u \in \mM(J,\boldsymbol{\gamma}^+,\boldsymbol{\gamma}^-)$.

A standard application of the implicit function theorem
implies that the open subset consisting of Fredholm regular curves in 
$\mM_m(J,\boldsymbol{\gamma}^+,\boldsymbol{\gamma}^-)$ admits the structure of
a smooth finite-dimensional orbifold whose dimension locally equals its
virtual dimension, and it is a manifold near any curve that is somewhere injective.
Moreover, a standard argument via the Sard-Smale theorem
(see \cite{McDuffSalamon:Jhol} or \cite{Wendl:lectures}) shows that
after perturbing $J$ generically in $\jJ(W,\omega,\alpha_+,\alpha_-)$
on some open subset $\uU \subset W$ with compact closure, one can assume
that all somewhere injective curves passing through $\uU$ are Fredholm regular.
Similarly, Dragnev \cite{Dragnev} (see also \cite{Wendl:blogForgetfulGood})
has shown that on a symplectization $(\RR \times M,d(e^r \alpha))$,
generic perturbations within $\jJ(\alpha_\pm)$ suffice to
make all somewhere injective curves regular, and this result can also be
applied to any curves in the cobordism $\overline{W}$ that are contained in
a cylindrical end.

If the Reeb flows on $M_+$ and $M_-$ are both globally nondegenerate or
Morse-Bott, then $\mM_m(J,\boldsymbol{\gamma}^+,\boldsymbol{\gamma}^-)$ has
a natural compactification
$$
\overline{\mM}_m(J,\boldsymbol{\gamma}^+,\boldsymbol{\gamma}^-)
$$
defined in \cite{SFTcompactness}, consisting
of \defin{stable holomorphic buildings} of arithmetic genus zero with
$m$ marked points.  An example of a holomorphic
building (with higher arithmetic genus) is shown
in Figure~\ref{fig:compactness}.
We shall write holomorphic buildings using the notation
$$
(v_{N_+}^+ | \ldots | v_1^+ | v_0 | v_1^- | \ldots | v_{N_-}^-),
$$
where $N_+ , N_- \ge 0$ are integers, $v_1^\pm,\ldots,v_{N_\pm}^\pm$ are each
(possibly disconnected and/or nodal) $J_\pm$-holomorphic curves in the symplectizations
$\RR \times M_\pm$, forming the \defin{upper} and \defin{lower levels}
respectively, and $v_0$ is a (possibly disconnected and/or nodal) $J$-holomorphic
curve in $\overline{W}$, the \defin{main level}.  Note that by convention,
the main level is allowed to be empty (i.e.~$v_0$ is a curve with domain the
empty set) if $N_+$ or $N_-$ is nonzero.  Each upper or level
is defined only up to $\RR$-translation, and the same is true of all levels
when $\overline{W}$ is a symplectization, in which case there is no 
distinguished ``main'' level or distinction between ``upper'' and ``lower''
levels.  The evaluation map extends continuously over
$\overline{\mM}_m(J,\boldsymbol{\gamma}^+,\boldsymbol{\gamma}^-)$ if we
also compactify $\overline{W}$ by adding $\{\pm\infty\} \times M_\pm$ to the top
and bottom of the cylindrical ends, i.e.~marked points in upper or lower
levels are mapped to $\{\infty\} \times M_+$ or $\{-\infty\} \times M_-$
respectively.

Our notation for buildings is convenient but suppresses
an additional detail that will sometimes be quite important: the data
also includes a one-to-one corresondence between the positive punctures of
each level (other than the topmost) and the negative punctures of the level
above it, such that corresponding punctures
have matching asymptotic orbits, the so-called \defin{breaking orbits}.  
Additionally, each pair of corresponding punctures is equipped
with a choice of a rotation angle for gluing the corresponding
positive and negative ends along the breaking orbit---this choice is
unique if the orbit is simple, but in general there are $m \in \NN$
distinct choices if the orbit has covering multiplicity~$m$.  All of this
data together is called a \defin{decoration} of the building.  Different
choices of decoration often produce buildings that are 
biholomorphically inequivalent 
to each other and thus represent distinct elements of
$\overline{\mM}_m(J,\boldsymbol{\gamma}^+,\boldsymbol{\gamma}^-)$.

\begin{figure}
\includegraphics[width=5in]{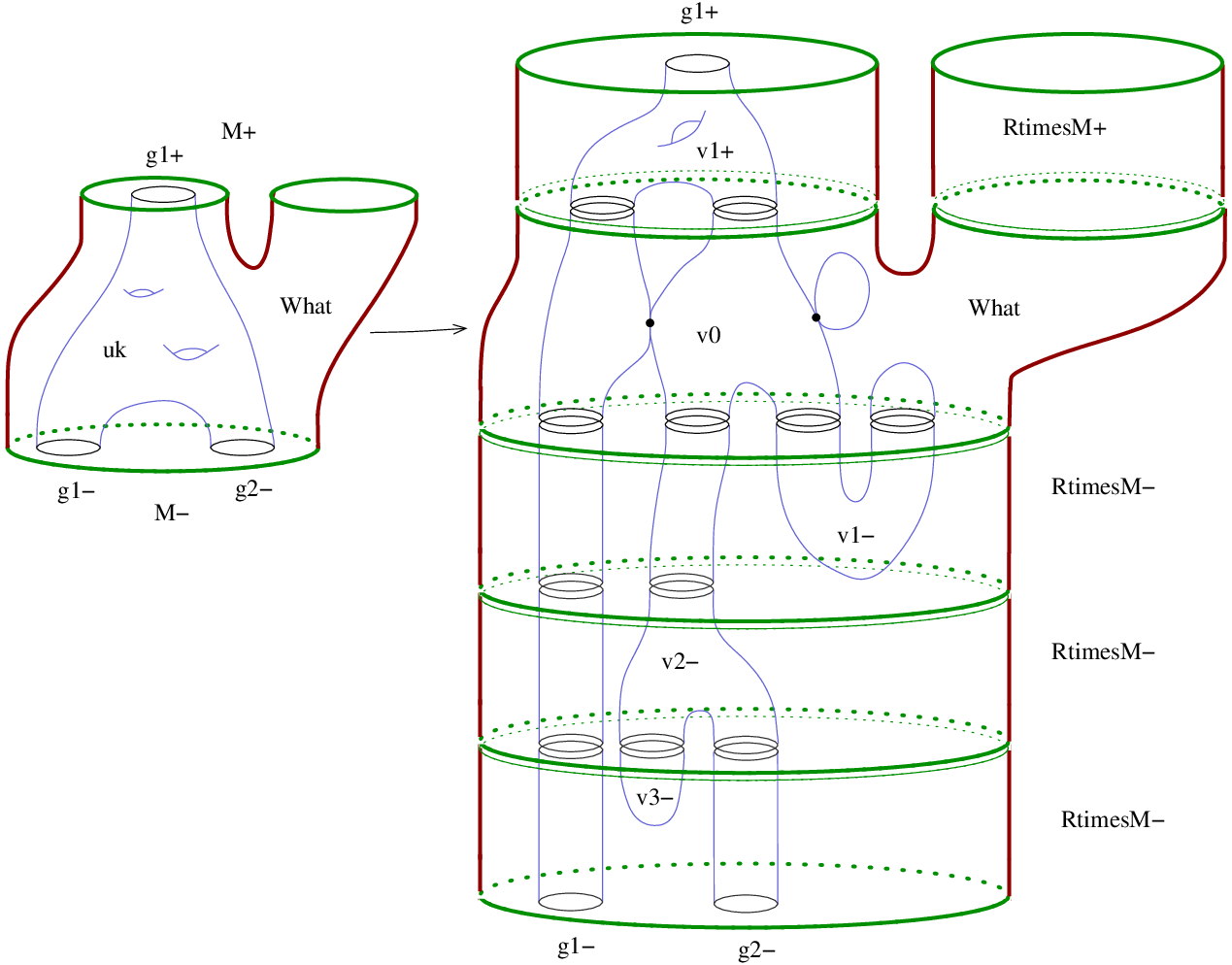}
\caption{\label{fig:compactness} The picture shows the degeneration of a 
sequence of punctured curves of genus 2 into a building with a main
level, one upper level and three lower levels.
We label the building as $(v_1^+|v_0|v_1^-|v_2^-|v_3^-)$, where each
$v_i^\pm$ is in general a disconnected nodal curve in a single level.
The arithmetic genus of the 
building is still 2, and the levels match along their respective 
asymptotic orbits.}
\end{figure}

Whenever $(W,\omega)$ is a Liouville cobordism (and in particular if
$\overline{W}$ is a symplectization), Stokes'
theorem prevents the existence of curves with no positive ends, sometimes
referred to as \defin{holomorphic caps}.
The following standard result is then immediate from the definition of
convergence in \cite{SFTcompactness}.

\begin{prop}
\label{prop:tree}
Suppose $\overline{W}$ is either a symplectization or the completion of a
Liouville cobordism, and
$u_k \in \mM_m(J,\gamma,\emptyset)$ is a sequence of $J$-holomorphic
planes converging to a holomorphic building.  Then the limiting building
has the following properties:
\begin{itemize}
\item Each connected component of each level is a punctured sphere with
precisely one positive puncture.
\item The lowest level has no negative punctures (so it is a disjoint union
of planes).
\item The top level is connected.
\item There are no nodes.
\end{itemize}
\qed
\end{prop}

We shall refer to the components without negative ends in the above lemma as
\defin{capping planes}; they are not to be confused with ``holomorphic
caps,'' which have \emph{only} negative ends.

The converse of compactness is gluing, as discussed e.g.~in
\cite{Nelson:thesis}*{Chapter~7}.  We will only need the following special case.

\begin{prop}
\label{prop:gluing}
Assume $\gamma_\infty$ is a Morse-Bott Reeb orbit in~$M_+$, $\gamma$ is a
nondegenerate orbit in~$M_-$, $m \ge 0$ is an integer, and
$(v_0 | v_1^-) \in \overline{\mM}_m(J,\gamma_\infty,\emptyset)$ is a
(decorated) stable $J$-holomorphic building such that
$v_0 \in \mM_m(J,\gamma_\infty,\gamma)$ and
$v_1^- \in \mM(J_-,\gamma,\emptyset) / \RR$ are both somewhere injective and
Fredholm regular.  Then there exist neighborhoods
\begin{equation*}
\begin{split}
v_0 &\in \uU_0 \subset \mM_m(J,\gamma_\infty,\gamma) \\
v_1^- &\in \uU_- \subset \mM(J_-,\gamma,\emptyset) / \RR
\end{split}
\end{equation*}
and a smooth embedding
$$
\Psi : [0,\infty) \times \uU_0 \times \uU_- \hookrightarrow \mM_m(J,\gamma_\infty,\emptyset)
$$
such that for any sequences $[0,\infty) \ni r_k \to +\infty$,
$u_k \to u_\infty \in \uU_0$ and $u_k^- \to u_\infty^- \in \uU_-$,
$$
\Psi(r_k,u_k,u_k^-) \to (u_\infty | u_\infty^-) \in \overline{\mM}_m(J,\gamma_\infty,\emptyset)
$$
in the SFT topology.  Moreover, every smooth curve in
$\mM_m(J,\gamma_\infty,\emptyset)$ sufficiently close to $(v_0 | v_1^-)$
in the SFT topology is in the image of~$\Psi$.
\qed
\end{prop}
\begin{remark}
The notation for buildings used in Proposition~\ref{prop:gluing} implicitly
assumes that if multiple buildings can be constructed out of
$u_\infty$ and $u_\infty^-$ via different choices of decoration, then
$(u_\infty | u_\infty^-)$ is the unique choice that is close to
$(v_0 | v_1^-)$ in the SFT topology.
\end{remark}

\subsection{The low-dimensional case}

We now specialize to the case where the cobordism $(W,\omega)$ is 
$4$-dimensional, so all contact manifolds under consideration will be
$3$-dimensional.

\subsubsection{Indices of covers}

We begin with a pair of convenient numerical observations.  The first is
borrowed (along with its proof) from \cite{Hutchings:index}.

\begin{prop}
\label{prop:trivCyls}
Suppose $J \in \jJ(\alpha)$ for a contact $3$-manifold $(M,\xi=\ker\alpha)$,
and $u : (\dot{\Sigma},j) \to (\RR \times M,J)$ is a $J$-holomorphic
branched cover of a trivial cylinder over a Reeb orbit whose covers are all
nondegenerate.  Then $\ind{u} \ge 0$, and equality can hold only when
the cover is unbranched or the orbit is elliptic.
\end{prop}
\begin{proof}
If the underlying orbit $\gamma$ is hyperbolic, then the index formula gives
$\ind{u} = -\chi(\dot{\Sigma}) \ge 0$ due to Lemma~\ref{czcover}, which is
an equality if and only if $\dot{\Sigma}$ is the cylinder, in which case
the Riemann-Hurwitz formula implies that the cover is unbranched.
If the orbit is instead elliptic, we can make our lives slightly easier with
the observation that $u$ has the same index as that of some
holomorphic building whose connected components are all thrice-punctured
spheres that are also branched covers of the same trivial cylinder.  It 
therefore suffices to prove that the inequality holds for thrice-punctured
spheres.  If for instance $u$ has two positive punctures at $\gamma^k$
and $\gamma^\ell$ and a negative puncture at $\gamma^{k+\ell}$, then
Lemma~\ref{czcover} gives
$$
\ind{u} = -\chi(\dot{\Sigma}) + \left(2\lfloor k\theta \rfloor + 1 \right)
+ \left(2 \lfloor \ell \theta \rfloor + 1 \right) -
\left(2 \lfloor (k+\ell)\theta \rfloor + 1 \right),
$$
where $\chi(\dot{\Sigma}) = -1$, and the index is thus
nonnegative due to the relation $\lfloor a + b \rfloor \le
\lfloor a \rfloor + \lfloor b \rfloor + 1$.  In the inverse case with
one positive puncture and two negative, we get the same result using
$\lfloor a \rfloor + \lfloor b \rfloor \le \lfloor a + b \rfloor$.
\end{proof}

\begin{prop}
\label{prop:coversHyp}
Suppose $\dim W = 4$ and $u = v \circ \phi : (\dot{\Sigma},j) \to (\overline{W},J)$
is a $k$-fold cover of a somewhere injective $J$-holomorphic curve
$v : (\dot{\Sigma}',j') \to (\overline{W},J)$ whose asymptotic orbits are
all nondegenerate and hyperbolic.  Then
$$
\ind{u} \ge k \ind{v},
$$
with equality if and only if the cover $\phi : (\dot{\Sigma},j) \to 
(\dot{\Sigma}',j')$ has no branch points in the punctured surface~$\dot{\Sigma}$.
\end{prop}
\begin{proof}
This is a direct consequence of the index formula \eqref{eqn:index}
together with Lemma~\ref{czcover} and the Riemann-Hurwitz formula
$Z(d\phi) = -\chi(\dot{\Sigma}) + k \chi(\dot{\Sigma}')$, where
$Z(d\phi) \ge 0$ denotes the algebraic count of zeroes of the holomorphic
section $d\phi \in \Gamma(\Hom_\CC(T\dot{\Sigma},\phi^*T\dot{\Sigma}'))$,
and thus vanishes if and only if the cover is unbranched.
\end{proof}

\subsubsection{Asymptotic defect}

Suppose $u : \dot{\Sigma} \to \overline{W}$ is asymptotic at $z \in \Gamma^\pm$
to a $T$-periodic orbit $\c : S^1 \to M_\pm$ and has an asymptotic
representative $h(s,t)$ at this puncture that is not identically zero.
Then the asymptotic formula \eqref{eqn:asympFormula} provides a nonzero
eigenfunction $e_1 \in \mc{C}^\i(\c^*\xi_\pm)$, and given a
trivialization~$\Phi$ of $\c^*\xi_\pm$, one can define
$$
\op{wind}_\i^\Phi(u;z) := \op{wind}^\Phi(e_1) \in \ZZ.
$$
If $z \in \Gamma^+$, then $\a_-^\Phi(\c)$ is the winding of the 
greatest negative eigenvalue of $A_{\c}$, thus
$\op{wind}_\i^\Phi(u;z)\leq \a_-^\Phi(\c)$, and similarly,
$\op{wind}_\i^\Phi(u;z) \ge \a_+^\Phi(\c)$ if $z \in \Gamma^-$.
The difference $\a_-^\Phi(\c)-\op{wind}_\i^\Phi(u;z)$ or
$\op{wind}_\i^\Phi(u;z) - \a_+^\Phi(\c)$ for a positive or negative puncture
respectively is denoted $d_0(u;z) \ge 0$ and called the \defin{asymptotic defect} 
of $u$ at $z \in \Gamma$.  Notice that it does not depend on the trivialization.
The total asymptotic defect of $u$ is then a nonnegative integer
$$
d_0(u)=\sum_{z \in \Gamma} d_0(u;z).
$$
This is well defined for any curve $u$ that is not identical to a trivial
cylinder in some neighbourhood of any of its punctures; in particular,
if $\overline{W}$ is a symplectization $(\RR \times M,d(e^r\alpha))$ with
$J \in \jJ(\alpha)$, then $d_0(u)$ is well defined for every curve other
than covers of trivial cylinders.

\subsubsection{The normal Chern number and $\windpi(u)$}

The \defin{normal Chern number} of a curve $u \in \mM(J,\boldsymbol{\gamma}^+,\boldsymbol{\gamma}^-)$
with all asymptotic orbits nondegenerate is defined by
$$
c_N(u)= c_1^\Phi(u^*T\overline{W}) - \chi(\dot{\Sigma}) +
\sum_{i=1}^{k_+} \a_-^\Phi(\gamma_i^+) -
\sum_{i=1}^{k_-} \a_+^\Phi(\gamma_i^-),
$$
where $\Phi$ is again an arbitrary choice of unitary trivializations
of $\xi_\pm$ along the asymptotic orbits, and the sum does not depend on this
choice.  The index formula and relations between Conley-Zehnder indices and
winding numbers imply
\begin{equation} \label{cneq}
2c_N(u)=\ind{u} -2+2g + \#\Gamma_0,
\end{equation}
where $g \ge 0$ is the genus of the domain (zero in our case) 
and $\Gamma_0 \subset \Gamma$ denotes the set of punctures of $u$ that
have even parity.  In the Morse-Bott setting of \eqref{eqn:indexMB},
the definition of $c_N(u)$ given above remains valid, and so does
\eqref{cneq} after interpreting $\Gamma_0$ as the set of
punctures for which the \emph{perturbed} Conley-Zehnder index 
(see \eqref{eqn:perturbedCZ}) is even.
One can interpret $c_N(u)$ as ``$c_1$ of the normal bundle''
when $u$ is immersed; in particular, $c_N(u)$ then predicts the number of 
zeroes for a generic section in the kernel of the linearized normal
deformation operator at~$u$, see e.g.~\cite{Wendl:automatic}.

For curves in the symplectization $\mb{R} \times M$ of a contact manifold 
$(M,\xi=\ker\alpha)$, there is a further invariant related to $c_N(u)$ and the 
asymptotic defect.  Let $\pi : TM \to \xi$ denote the fibrewise linear 
projection along the
Reeb vector field.  Then the nonlinear Cauchy-Riemann equation for 
$u : \dot{\Sigma} \to \RR \times M$ implies that
$\pi \circ du \in \mc{C}^\infty(\Hom_\CC(T\dot{\Sigma},u^*\xi))$ locally
satisfies a linear Cauchy-Riemann type equation, so zeroes of $\pi \circ du$
are isolated and positive by the similarity principle unless
$\pi \circ du \equiv 0$.  The latter is the case if and only if $u$ is a
cover of a trivial cylinder, and otherwise, we define
$$
\windpi(u) \ge 0
$$
to be the algebraic count of zeroes of $\pi \circ du$.  The asymptotic
formula \eqref{eqn:asympFormula} implies that zeroes of $\pi \circ du$ cannot
accumulate near infinity, so $\windpi(u)$ is always finite.
It equals $0$ if and only if $u = (u_\RR,u_M) : \dot{\Sigma} \to \RR \times M$
has the property that $u_M : \dot{\Sigma} \to M$ is an immersion transverse
to the Reeb vector field.
From \cite{HWZ:props2}*{Prop.~5.6}, we have
\begin{equation}
\label{eqn:windpiRelation}
c_N(u) = \windpi(u) + d_0(u).
\end{equation}
In particular, this implies
\begin{equation} \label{cnineq}
c_N(u)\geq d_0(u) \geq 0 \quad \text{ and } \quad
c_N(u) \ge \windpi(u) \ge 0
\end{equation}
for any curve that is not a cover of a trivial cylinder, so $c_N(u) = 0$
gives a homotopy-invariant sufficient condition for both the asymptotic
defect and $\windpi(u)$ to vanish.

\subsubsection{Self-linking numbers}
Let $\c$ be a nullhomologous transverse knot in a closed contact 
$3$-manifold $(M,\xi)$, let $\Sigma \subset M$ be a Seifert surface and $X$ 
a framing of $\c$, i.e.~a non-zero section of $\c^*\xi$.
The \defin{self-linking number} of $\c$ with respect to $X$ is 
then the  algebraic count of intersections between $\Sigma$ and a generic 
push-off of $\gamma$ in the direction of~$X$:
$$
\selflinking(\c,X) = (\exp_\c X) \cdot \Sigma \in \ZZ.
$$ 
Note that this depends on $X$ up to homotopy, but not on $\Sigma$, as a
different choice of Seifert surface changes $\selflinking(\c,X)$
by the homological intersection number of $\c$ with a closed $2$-cycle,
which vanishes since $\c$ is nullhomologous.
Replacing $X$ with another framing changes $\selflinking(\c,X)$ by the 
relative winding of the two framings,
\begin{equation}
\selflinking(\c, X_1) - \selflinking(\c, X_2) = \op{wind}(X_1, X_2),
\end{equation}
where $\wind(X_1,X_2) \in \ZZ$ denotes the winding number of the section $X_1$
along $\gamma$ in the trivialization induced by~$X_2$.
Note that the Seifert surface determines a canonical homotopy class of 
framings $X_\Sigma$ via the condition that $X_\Sigma$ should extend to a trivialization 
of $\xi$ along~$\Sigma$, so with this choice we shall denote
$$
\selflinking(\c;\Sigma) := \selflinking(\c,X_\Sigma).
$$
This depends on $\Sigma$ since $X_\Sigma$ does, but the dependence 
vanishes if $c_1(\xi) = 0$.

With this definition in mind, suppose $\c$ is an unknotted Reeb orbit 
and $u = (u_\RR,u_M) : \mb{C} \rightarrow \mb{R} \times M$ is a $J$-holomorphic 
plane asymptotic to $\c$ for which $u_M : \CC \to M$ is embedded. 
The closure of $u_M(\CC)$ is then a Seifert disk $\dD \subset M$ for~$\gamma$,
and we claim
\begin{equation}\label{selflinking}
\selflinking(\c;\dD)=\op{wind}(X_\dD, e_1(u)),
\end{equation}
where $X_\dD$ is the canonical framing determined by $\dD$ as discussed above,
and $e_1(u)$ is the nonzero eigenfunction appearing in the asymptotic formula
\eqref{eqn:asympFormula} for the approach of $u$ to~$\gamma$.
Indeed, $e_1(u)$ gives the direction of the approach of $u$ to~$\gamma$ and
is thus homotopic to the Seifert framing of~$\gamma$, implying
$\selflinking(\c,e_1(u)) = 0$, so
$$
\selflinking(\c;\dD)=\selflinking(\c, X_\dD)=\selflinking(\c,e_1(u))+ \op{wind}(X_\dD, e_1(u))= \op{wind}(X_\dD, e_1(u)).
$$

\subsubsection{Siefring intersection theory}
\label{sec:intersection}

We recall here some useful properties of the intersection product on classes 
of $J$-holomorphic curves in almost complex manifolds with cylindrical ends.
In \cite{Siefring:intersection}, Siefring associates to any pair of
(not necessarily $J$-holomorphic) asymptotically cylindrical maps 
$u_1 : \dot{\Sigma}_1 \to \overline{W}$ and $u_2 : \dot{\Sigma}_2 \to \overline{W}$
with nondegenerate or Morse-Bott asymptotic orbits an integer
$$
u_1 * u_2 \in \ZZ,
$$
which matches the homological intersection number $[u_1] \cdot [u_2]$ if
both curves have no punctures, and in general has the following properties.
First, the pairing is symmetric
$$
u_1 * u_2 = u_2 * u_1,
$$
and it is invariant under homotopies of asymptotically cylindrical maps with
fixed asymptotic orbits;
in fact, $u_1 * u_2$ depends only on the asymptotic orbits of $u_1$ and $u_2$ and
their relative homology classes.  If both maps are $J$-holomorphic and their
images are non-identical, then the relative asymptotic results of
\cite{Siefring:asymptotics} imply that all intersections between $u_1$ and
$u_2$ are isolated and contained in a compact subset, so by positivity
of intersections, the algebraic count of intersections $u_1 \cdot u_2$ is
finite and satisfies
$$
u_1 \cdot u_2 \ge \left| \left\{ (z_1,z_2) \in \dot{\Sigma}_1 \times \dot{\Sigma}_2\ 
\Big|\ u_1(z_1) = u_2(z_2) \right\}\right|,
$$
with equality if and only if all intersections are transverse.  This is then
related to $u_1 * u_2$ by
$$
u_1 * u_2 \ge u_1 \cdot u_2,
$$
so the condition $u_1 * u_2 = 0$ gives a homotopy-invariant sufficient 
condition for $u_1$ and $u_2$ to be disjoint.  The following computation is
an easy consequence of the definition (cf.~\cite{Siefring:intersection}*{Prop.~5.6}):

\begin{prop}
\label{prop:evenCylinder}
Suppose $J \in \jJ(\alpha)$ for a contact $3$-manifold $(M,\xi=\ker\alpha)$,
and $u$ and $v$ are both $J$-holomorphic covers of the same trivial cylinder 
in $(\RR \times M,J)$ over a nondegenerate Reeb orbit with even parity.
Then $u*v = 0$.
\qed
\end{prop}

The intersection product also has a natural extension to holomorphic
buildings such that homotopy invariance holds for all continuous deformations
in the SFT topology.  We will need a particular result about this extension:

\begin{prop}
\label{prop:intBuilding}
If $v = (v_{N_+}^+ | \ldots | v_1^+ | v_0 | v_1^- | \ldots | v_{N_-}^-)$ is a
holomorphic building in a $4$-dimensional completed symplectic cobordism, we have
\begin{equation*}
\begin{split}
v * v \ge \sum_{j=1}^{N_+} v_j^+ * v_j^+ + v_0 * v_0 +
\sum_{j=1}^{N_-} v_j^- * v_j^- + \sum_{\gamma} m(\gamma) p(\gamma),
\end{split}
\end{equation*}
where the last sum is over all orbits $\gamma$ that occur as breaking orbits
in~$v$, with covering multiplicities denoted by $m(\gamma) \in \NN$.
\end{prop}
\begin{proof}
The existence of a formula
$$
v * v = \sum_{j=1}^{N_+} v_j^+ * v_j^+ + v_0 * v_0 +
\sum_{j=1}^{N_-} v_j^- * v_j^- + Q
$$
with some error term $Q \ge 0$ is stated in 
\cite{Siefring:intersection}*{Prop.~4.3(4)}, and our lower bound on the error term
can be extracted from the proof of that result.  The point is the following.
Using notation from \cite{Wendl:Durham}, the pairing $u*w$ can be written
in general as
$$
u * w = u \bullet_\Phi w - \sum_{(z,\zeta)} \Omega^\Phi_+(\gamma_z,\gamma_\zeta) -
\sum_{(z,\zeta)} \Omega^\Phi_-(\gamma_z,\gamma_\zeta),
$$
where $u \bullet_\Phi w \in \ZZ$ denotes the relative intersection number of $u$ and
$w$ with respect to an arbitrarily chosen asymptotic trivialization~$\Phi$,
the two sums are over all pairs of positive resp.~negative punctures 
$z$ of $u$ and $\zeta$ of~$w$, $\gamma_z$ and $\gamma_\zeta$ are the
corresponding asymptotic orbits, and $\Omega^\Phi_\pm(\gamma_z,\gamma_\zeta)$
are integers determined by the winding numbers $\a^\Phi_\mp(\gamma_z)$
and $\a^\Phi_\mp(\gamma_\zeta)$ (see \cite{Wendl:Durham}*{\S 4.2} for a precise formula).
The same formula for $u*w$ is valid if $u$ and $w$ are buildings, and the
relative intersection numbers are additive across levels.  The difference
between $v*v$ and the sum of the invariant self-intersection numbers of its
levels is therefore a sum of terms of the form
$\Omega^\Phi_+(\gamma,\gamma') + \Omega^\Phi_-(\gamma,\gamma')$ where
$\gamma$ and $\gamma'$ are pairs of breaking orbits of~$v$.  All of these
terms are nonnegative, and in particular whenever $\gamma^m$ is a breaking
orbit (with $\gamma$ denoting the underlying simply covered orbit), they include
$$
\Omega^\Phi_+(\gamma^m,\gamma^m) + \Omega^\Phi_-(\gamma^m,\gamma^m) =
m \a^\Phi_+(\gamma^m) - m \a^\Phi_-(\gamma^m) = m p(\gamma^m).
$$
\end{proof}

If $u : \dot{\Sigma} \to \overline{W}$ is somewhere injective and
$J$-holomorphic, then the relative asymptotic results of \cite{Siefring:asymptotics}
also imply that it is embedded outside a compact subset, so there is a
finite singularity count $\delta(u) \in \ZZ$, defined as the algebraic
count of double points $\{ (z_1,z_2) \in \dot{\Sigma} \times \dot{\Sigma}\ |\  
u(z_1) = u(z_2) \text{ and } z_1 \ne z_2 \}$ after perturbing $u$ in a compact
subset to make it immersed.  
Standard local results due to Micallef and White \cite{MicallefWhite}
imply that $\delta(u) \ge 0$ with equality if and only if $u$ is embedded,
but in contrast to the closed case, $\delta(u)$ is not generally homotopy invariant.
Instead, it satisfies the generalized adjunction formula
\begin{equation}
\label{eqn:adjunction}
u * u = 2 \dtotal(u) + c_N(u) + \left[ \bar{\sigma}(u) - \#\Gamma \right],
\end{equation}
where
$$
\dtotal(u) = \delta(u) + \dinfty(u)
$$
includes an additional contribution $\dinfty(u) \ge 0$ counting ``hidden''
double points that can emerge from infinity under generic perturbations,
and the term $\bar{\sigma}(u) \in \NN$ is a sum of the spectral covering
numbers (see \S\ref{sec:Reeb}) of all asymptotic orbits, hence
$\bar{\sigma}(u) - \#\Gamma$ is also nonnegative.
The formula implies that $\dtotal(u)$ is homotopy invariant, and since
$\dinfty(u) \ge 0$, the condition $\dtotal(u) = 0$ then suffices to
ensure that all somewhere injective curves homotopic to $u$ are embedded.
The converse is false in general: a curve can still be embedded with
$\dtotal(u) > 0$ due to hidden intersections, which can emerge from infinity
under perturbations---but this can only happen if $u$ has at least one
multiply covered asymptotic orbit or at least two punctures of the same sign
that approach covers of the same orbit, thus giving the following useful
criterion:

\begin{lem}
\label{lemma:dinfty}
If $u$ is a somewhere injective curve whose asymptotic orbits are all
distinct and simple, then $\dinfty(u) = \bar{\sigma}(u) - \#\Gamma = 0$.
\qed
\end{lem}

The following is a minor improvement on a definition originating in
\cites{Wendl:compactnessRinvt,Wendl:automatic}.

\begin{defn}\label{ne}
An asymptotically cylindrical $J$-holomorphic curve $u : \dot{\Sigma} \to \overline{W}$
is called \defin{nicely embedded} if it is somewhere injective and satisfies
$u * u \le 0$ and $\dtotal(u) = 0$.
\end{defn}

It is clear from the above discussion that if $u$ is nicely embedded, then
so is any other somewhere injective curve $u'$ in the same component of the
moduli space, and moreover, $u$ and $u'$ must then be disjoint.  Nicely embedded
curves arise naturally in the study of finite energy foliations,
initiated in \cite{HWZ:foliations}.  Their most important properties for our
purposes are the following.

\begin{lem}
\label{lemma:neIndex}
If $u \in \mM(J,\boldsymbol{\gamma}^+,\boldsymbol{\gamma}^-)$ is nicely
embedded then $c_N(u) \le 0$ and $\ind{u} \le 2$.
\end{lem}
\begin{proof}
The first inequality follows directly from the definition and the 
adjunction formula \eqref{eqn:adjunction}, and this implies the second
via \eqref{cneq}.
\end{proof}

\begin{prop}
\label{prop:automatic}
If $u \in \mM(J,\boldsymbol{\gamma}^+,\boldsymbol{\gamma}^-)$ is a nicely
embedded curve with $\ind{u} \in \{1,2\}$, then $u$ is Fredholm regular.
\end{prop}
\begin{proof}
Since $u$ is immersed by assumption and, by Lemma~\ref{lemma:neIndex},
satisfies $c_N(u) \le 0$, it satisfies the criterion $\ind{u} > c_N(u)$
for automatic transversality given in \cite{Wendl:automatic}.
\end{proof}

\begin{prop}
\label{prop:localFol}
Suppose $\mM\nice \subset \mM_1(J,\boldsymbol{\gamma}^+,\boldsymbol{\gamma}^-)$ 
is an open and closed subset of the space of nicely embedded 
index~$2$ curves, equipped with the extra data of a marked point, such that
all curves in $\mM\nice$ represent the same relative homology class.
Then $\mM\nice$ is a smooth $4$-manifold, and the evaluation map
$$
\ev : \mM\nice \to \overline{W}
$$
is an embedding onto an open subset of~$\overline{W}$.
\end{prop}
\begin{proof}
This is a mild generalization of a similar result proved in \cite{HWZ:props3}
for planes with simply covered asymptotic orbits.  We know every $u \in \mM\nice$ is Fredholm
regular by Prop.~\ref{prop:automatic}, and $c_N(u)=0$ due to \eqref{cneq}
and Lemma~\ref{lemma:neIndex}.  It follows that $\mM\nice$ is
smooth and has dimension $\ind{u}+ 2 = 4$, and since $u*u \le 0$ 
(which becomes $u*u = 0$ when $c_N(u)=0$), invariance of the
intersection number implies that no two curves in $\mM\nice$ can intersect,
hence $\ev : \mM\nice \to \overline{W}$ is injective.  To see that
it is also an immersion, observe that for a given curve
$u_0 : \dot{\Sigma} \to \overline{W}$ and marked point $\zeta_0 \in \dot{\Sigma}$
with the pair $(u_0,\zeta_0)$ representing an element of~$\mM\nice$, the
tangent space $T_{(u_0,\zeta_0)}\mM\nice$ is naturally
identified with the direct sum of $T_{\zeta_0} \dot{\Sigma}$ and the
kernel of the linearized Cauchy-Riemann operator acting
on the normal bundle of~$u_0$.  The condition $c_N(u_0)=0$ then implies via 
\cite{Wendl:automatic}*{Equation~(2.7)} that sections in this kernel are
nowhere zero, hence the derivative of the evaluation map $\ev(u,\zeta) =
u(\zeta)$ at $(u_0,\zeta_0)$ is injective.
\end{proof}

\begin{prop}
\label{prop:ne}
Suppose $\overline{W}$ is a symplectization $(\RR \times M,d(e^r\alpha))$ and
$J \in \jJ(\alpha)$.  Then for any nicely embedded $J$-holomorphic curve
$u = (u_\RR,u_M) : \dot{\Sigma} \to \RR \times M$ that is not a trivial
cylinder, the map $u_M : \dot{\Sigma} \to M$ is embedded.
\end{prop}
\begin{proof}
Since $c_N(u) \le 0$ by Lemma~\ref{lemma:neIndex},
$\windpi(u) = 0$ due to \eqref{eqn:windpiRelation} and $u_M$ is
therefore immersed and transverse to the Reeb vector field.  To show that
$u_M$ is injective, observe that any double point $u_M(z_1) = u_M(z_2)$
can be interpreted as an intersection of $u$ with one of its $\RR$-translations
$u^c := (u_\RR + c,u_M)$ for some $c \in \RR$, and $c$ must be nonzero
since $\dtotal(u)=0$ implies that $u$ itself is embedded.  By homotopy
invariance of the intersection product, $u * u = u * u^c \le 0$, so such an
intersection is possible only if $u$ and $u^c$ are the same curve up to
parametrization.  But this would imply that $u$ is also equivalent to
$u^{kc}$ for every $k \in \NN$, so taking $k \to \infty$, we conclude from 
the asymptotically cylindrical behaviour of $u$ that its image lies in an
arbitrarily small neighbourhood of a collection of trivial cylinders.
This can only happen if $u$ itself is a trivial cylinder, so we have a
contradiction.
\end{proof}

\begin{lem}
\label{lemma:selfLinking}
Under the assumptions of Prop.~\ref{prop:ne}, suppose
$u = (u_\RR,u_M) : \CC \to \RR \times M$ is a nicely embedded plane
asymptotic to a simply covered orbit~$\gamma$ and $\ind{u} \in \{1,2\}$.
Then if $\dD \subset M$ denotes the Seifert surface with interior $u_M(\CC)$, we have
$$
\muCZ(\gamma;\dD) = \begin{cases}
2 & \text{ if $\ind{u}=1$,}\\
3 & \text{ if $\ind{u}=2$,}
\end{cases}
$$
and in both cases $\selflinking(\gamma ; \dD) = -1$.
\end{lem}
\begin{proof}
If $\Phi$ is the trivialization of $\gamma^*\xi$ that extends over~$\dD$,
then the relative $c_1$ term in the index formula vanishes and gives the stated
relation between $\ind{u}$ and $\muCZ^\Phi(\gamma)$.  By \eqref{eqn:CZwinding}
and \eqref{eqn:parity},
this implies $\a^\Phi_-(\gamma) = 1$.  Moreover, $c_N(u) \le 0$ by
Lemma~\ref{lemma:neIndex}, thus \eqref{cnineq} implies that $u$ has zero 
asymptotic defect, so the nonzero eigenfunction $e_1(u)$ appearing in the
asymptotic formula \eqref{eqn:asympFormula} satisfies
$$
\wind^\Phi(e_1(u)) = \a^\Phi_-(\gamma) = 1.
$$
Now by \eqref{selflinking},
$$
\selflinking(\gamma;\dD) = -\wind^\Phi(e_1(u)) = -1.
$$
\end{proof}

\section{Seed curves in the positive end}
\label{sec:seed}

In this section we describe the seed curves that will generate the moduli
spaces required for proving Theorems~\ref{thm:main}, \ref{thm:nonprime}
and~\ref{thm:allDimensions}.

\subsection{The standard sphere}
\label{sec:stdSphere}

The following construction is for the proofs of Theorems~\ref{thm:main}
and~\ref{thm:allDimensions}.

Regarding $S^{2n-1}$ as the unit sphere in $\CC^n$, fix the standard contact 
form $\alpha\std$ described in Example~\ref{ex:stdSphere}, along with the 
unique admissible complex structure $J\std \in \jJ(\alpha\std)$ on
$\RR \times S^{2n-1}$ that 
restricts to $\xi\std \subset T S^{2n-1} \subset \CC^n$ as the standard
complex structure~$i$.  Recall that the diffeomorphism
\begin{equation}
\label{eqn:cyl}
(\RR \times S^{2n-1},J\std) \to (\CC^n \setminus \{0\},i) : (r,x) \mapsto e^{2r} x
\end{equation}
is then biholomorphic, so we can regard holomorphic curves in 
$\CC^n \setminus \{0\}$ as $J\std$-holomorphic curves in the 
symplectization of $(S^{2n-1},\xi\std)$.  With this understood, define for
each $w \in \CC^{n-1} \setminus \{0\}$ the holomorphic plane
$$
u_w : (\CC,i) \to (\CC^n \setminus \{0\},i) : z \mapsto (z,w).
$$
As a curve in $\RR \times S^{2n-1}$, each $u_w$ is asymptotic at $\infty$ 
to the same closed Reeb orbit in $(S^{2n-1},\alpha\std)$, namely
$$
\gamma_\infty : S^1 \to S^{2n-1} : t \mapsto (e^{2\pi i t},0,\ldots,0).
$$
This orbit belongs to a $(2n-2)$-dimensional Morse-Bott family of closed
embedded Reeb orbits with period~$\pi$, which foliate $S^{2n-1}$;
indeed, they form the fibres of the Hopf fibration $S^1 \hookrightarrow 
S^{2n-1} \to \CC P^{n-1}$.

\begin{lem}
\label{lemma:uwIndex}
For each $w \in \CC^{n-1} \setminus \{0\}$, $\ind{u_w} = 2n - 2$.
\end{lem}
\begin{proof}
Abbreviate $\overline{W} = \RR \times S^{2n-1}$.  The fibres of the contact
bundle along $\gamma_\infty$ are naturally identified with
$\{0\} \oplus \CC^{n-1} \subset
T S^{2n-1} \subset \CC^n$, so $\gamma_\infty^*\xi\std$ has a natural 
trivialization, which we will denote by~$\Phi$, and it extends to a natural 
trivialization of the normal bundle $N_{u_w} \to \CC$ of~$u_w$.  The latter 
implies $c_1^\Phi(N_{u_w}) = 0$, so writing $u_w^*T\overline{W} = T\CC \oplus N_{u_w}$
gives
$$
c_1^\Phi(u_w^*T\overline{W}) = \chi(\CC) + c_1^\Phi(N_{u_w}) = 1.
$$
To compute $\cz{\gamma_\infty + \epsilon}{\Phi}$, we observe that the
asymptotic operator $A_{\gamma_\infty}$ splits with respect to the obvious
decomposition
$$
\gamma_\infty^*\xi\std = S^1 \times \CC^{n-1} = L_2 \oplus \ldots \oplus L_n
$$
into trivial complex line bundles, so we can write
$A_{\gamma_\infty} = A_2 \oplus \ldots \oplus A_m$,
and the trivialization $\Phi$ is now also a direct sum
$\Phi_2 \oplus \ldots \oplus \Phi_m$ of trivializations of these line bundles.
The kernel of $A_{\gamma_\infty}$ is a complex $(n-1)$-dimensional space of
sections along $\gamma_\infty$ that point in the directions of other
Hopf fibres, and its intersection with each of the summands $L_j$ for
$j=2,\ldots,n$ is a complex $1$-dimensional space spanned by a section
of the form
$$
\eta_j : S^1 \to L_j : t \mapsto (0,\ldots,0,e^{2 \pi it},0,\ldots,0).
$$
We thus have $\wind^{\Phi_j}(\eta_j) = 1$, and $A_j + \epsilon$ therefore
has a real $2$-dimensional eigenspace with the smallest positive eigenvalue
$\epsilon$ and winding~$1$.  By Proposition~\ref{spectral}, the largest negative
eigenvalue $A_j + \epsilon$ must then have winding~$0$, so by
\eqref{eqn:CZwinding},
$$
\cz{A_j + \epsilon}{\Phi_j} = 1,
$$
and \eqref{eqn:CZhigherDim} then implies
$$
\cz{\gamma_\infty + \epsilon}{\Phi} = \sum_{j=2}^n \cz{A_j + \epsilon}{\Phi_j} =
n-1.
$$
Finally, \eqref{eqn:indexMB} gives
\begin{equation*}
\begin{split}
\ind{u_w} &= (n-3) \chi(\CC) + 2 c_1^\Phi(u_w^*T\overline{W}) +
\cz{\gamma_\infty + \epsilon}{\Phi} \\ &= (n-3) + 2 + (n-1)
= 2n - 2.
\end{split}
\end{equation*}
\end{proof}

\begin{lem}
\label{lemma:uwReg}
The $J\std$-holomorphic planes $u_w$ are all Fredholm regular.
\end{lem}
\begin{proof}
Note that the standard genericity arguments do not apply here since $J\std$
is very far from being generic.  But in this case we can check regularity
explicitly.  Recall that by \cite{Wendl:automatic}*{Theorem~3}, it suffices
to check that the linearized normal operator
$$
\mathbf{D}_{u_w}^N : W^{1,p,\delta}(N_{u_w}) \to L^{p,\delta}(\overline{\Hom}_\CC(T\CC,N_{u_w}))
$$
is surjective, where $\ind{\mathbf{D}_{u_w}^N} = \ind{u_w}$.
Here $p \in (2,\infty)$, and $\delta > 0$ is a small exponential weight,
meaning that if sections $\eta : \CC \to N_{u_w}$ in the domain of
$\mathbf{D}_{u_w}^N$ are written near $\infty$ in cylindrical coordinates
$(s,t) \in [0,\infty) \times S^1$ corresponding to $z = e^{2\pi (s+it)} \in \CC$,
then the section $e^{\delta s} \eta(s,t)$ must be of class $W^{1,p}$
on $[0,\infty) \times S^1$.  This definition also assumes a 
translation-invariant metric on $\RR \times S^{2n-1}$ for computing
$L^p$-norms of sections along~$u_w$.  Note that since $p > 2$, sections of
class $W^{1,p}$ are continuous, and we can therefore assume
\begin{equation}
\label{eqn:expDecay}
\eta(s,t) \to 0 \quad \text{ as } \quad s \to \infty
\end{equation}
for $\eta \in W^{1,p,\delta}(N_{u_w})$.

From a different perspective, however,
$\mathbf{D}_{u_w}^N$ is an extremely simple operator: sections $\eta$ of the 
normal bundle to $u_w : \CC \to \CC^n \setminus \{0\}$ can be identified 
canonically with functions $\tilde{\eta} : \CC \to \CC^{n-1}$ using the obvious
trivialization of $N_{u_w}$, and since $\mathbf{D}_{u_w}^N$ is the
linearization of the standard (and thus already linear) Cauchy-Riemann 
operator~$\dbar$, $\eta \in \ker \mathbf{D}_{u_w}^N$ implies that
$\tilde{\eta}$ is a $\CC^{n-1}$-valued holomorphic function.  Under the
transformation \eqref{eqn:cyl}, the condition \eqref{eqn:expDecay} then
implies
$$
\frac{|\tilde{\eta}(z)|}{|z|} \to 0 \quad \text{ as } \quad z \to \infty,
$$
so the growth of $\tilde{\eta}$ at infinity is strictly smaller than that of an
affine function.  Complex analysis then implies that the singularity
of $\tilde{\eta}$ at $\infty$ is removable, so $\tilde{\eta}$ is constant, proving
$$
\dim_\CC \ker \mathbf{D}_{u_w}^N = n-1.
$$
The real dimension of the kernel of $\mathbf{D}_{u_w}^N$ is thus equal
to its index according to Lemma~\ref{lemma:uwIndex}, so
$\mathbf{D}_{u_w}^N$ has trivial cokernel.
\end{proof}

\begin{lem}
\label{lemma:uwUniqueness}
Up to parametrization, every asymptotically cylindrical $J\std$-holomorphic 
curve in $\RR \times S^{2n-1}$ with a single positive puncture asymptotic 
to $\gamma_\infty$ and arbitrary negative punctures is either the trivial 
cylinder over $\gamma_\infty$ or one of the planes $u_w$.
\end{lem}
\begin{proof}
Since no Reeb orbit in $(S^3,\alpha\std)$ has period smaller than that
of~$\gamma_\infty$, any curve $u : \dot{\Sigma} \to \RR \times S^{2n-1}$
of the specified type with a nonempty set of negative punctures would satisfy 
$\int_{\dot{\Sigma}} u^*\alpha\std = 0$ by Stokes' theorem, and since
the positive asymptotic orbit is simple, $u$ in this case could only be
a trivial cylinder.  If $u$ has no negative punctures, then it defines via
\eqref{eqn:cyl} a proper holomorphic map $u = (u_1,\ldots,u_n) : 
\dot{\Sigma} \to \CC^n$ such that $u_2,\ldots,u_n : \dot{\Sigma} \to \CC$ 
are all bounded holomorphic functions that decay to $0$ at the unique puncture, 
so these all extend to holomorphic functions on the compact domain
$\Sigma$ and are therefore constant.  The remaining function $u_1 : \dot{\Sigma}
\to \CC$ has a pole of order~$1$ at its unique puncture, thus it extends to
a nonconstant holomorphic map $\Sigma \to S^2$ of degree~$1$, implying
that $\Sigma = S^2$ and, after a suitable reparametrization, $\dot{\Sigma} = \CC$
with $u_1 : \CC \to \CC$ an affine map.
\end{proof}

\begin{lem}
\label{lemma:uwNe}
In the case $\dim M = 3$, the planes $u_w$ satisfy $c_N(u_w) = 0$ and are
nicely embedded.
\end{lem}
\begin{proof}
We saw in the proof of Lemma~\ref{lemma:uwIndex} that 
$\cz{\gamma_\infty + \epsilon}{\Phi}$ is odd and $\ind{u_w} = 2$, so
\eqref{cneq} implies $c_N(u_w) = 0$.  Since $u_w$ is embedded and has only
a single simple asymptotic orbit, $\dtotal(u_w) = \bar{\sigma}(u_w) - 1 = 0$
by Lemma~\ref{lemma:dinfty}.  Thus by Siefring's adjunction formula \eqref{eqn:adjunction},
$$
u_w * u_w = 2 \dtotal(u_w) + c_N(u_w) + \left[ \bar{\sigma}(u_w) - 1 \right] = 0.
$$
\end{proof}

\subsection{Reducible tight contact $3$-manifolds}
\label{sec:reducible}

We now describe the seed curves for the first case of Theorem~\ref{thm:nonprime}.
Assume $M$ is a reducible closed oriented $3$-manifold with a
contact structure $\xi$; we are free to assume $\xi$ is tight since the
overtwisted case will be dealt with separately in \S\ref{sec:OTcase} below.
The reducibility hypothesis means that $M$ is either $S^1 \times S^2$
or a nontrivial connected sum $M_1 \# M_2$, and in the latter case,
tightness of $\xi$ implies via Colin's connected sum theorem \cite{Colin:prime}
that $(M,\xi) = (M_1,\xi_1) \# (M_2,\xi_2)$ for some tight contact
structures $\xi_i$ on $M_i$, $i=1,2$.  The case $S^1 \times S^2$ can also
be understood via connected sums since the unique tight contact structure
on $S^1 \times S^2$ is the one that is obtained from $(S^3,\xi\std)$ by
attaching two disjoint neighborhoods in $S^3$ to each other via a
self-connected sum.  In either case, $(M,\xi)$ contains a special embedded
$2$-sphere
$$
S^2 \cong S \subset M,
$$
the belt sphere of the connected sum, and $\xi$ takes a certain standard form
in a neighbourhood of~$S$.  Moreover, $S$ represents a nontrivial
element of $\pi_2(M)$: this follows from the Poincar\'e conjecture after
applying \cite{Hatcher:3manifolds}*{Prop.~3.1} to deduce that
$[S] \in \pi_2(M)$ can be trivial only if $S$ bounds a contractible 
submanifold in~$M$.

The desired $J$-holomorphic curves in $\RR \times M$ can now be borrowed 
wholesale from a construction of Fish and Siefring \cite{FishSiefring:1}.
Specifically, Theorem~5.1 in their paper provides a nondegenerate
contact form $\alpha_+$ on $(M,\xi)$ and an almost complex structure 
$J_+ \in \jJ(\alpha_+)$, which may be assumed generic outside a
neighborhood of $\RR \times S$, such that there exists
a nondegenerate embedded Reeb orbit 
$$
\gamma_\infty : S^1 \to M
$$
with even Conley-Zehnder index and with image in~$S$.  This orbit
splits $S$ into two hemispheres $S_+$ and $S_-$, and there exists a pair of
nicely embedded and Fredholm regular $J_+$-holomorphic planes
$$
u^\pm = (u^\pm_\RR,u^\pm_M) : \CC \to \RR \times M
$$
with index~$1$, both asymptotic to~$\gamma_\infty$, such that 
$u^\pm_M(\CC)$ is the interior of~$S_\pm$.  They satisfy
$$
c_N(u^\pm) = u^\pm * u^\pm = u^+ * u^- = 0,
$$
and they approach $\gamma_\infty$ ``from opposite sides'' in the sense that
after suitable $\RR$-translations, one can arrange
$$
e_1(u^+) = - e_1(u^-),
$$
where $e_1(u^\pm)$ denotes the nontrivial asymptotic eigenfunction appearing
in the asymptotic formula \eqref{eqn:asympFormula} for~$u^\pm$.

\begin{lem}
\label{lemma:csUniqueness}
Up to parametrization and $\RR$-translation, every asymptotically cylindrical
$J_+$-holomorphic curve in $\RR \times M$ with a single positive puncture
asymptotic to $\gamma_\infty$ and arbitrary negative punctures is either
the trivial cylinder over $\gamma_\infty$ or one of the planes~$u^\pm$.
\end{lem}
\begin{proof}
We use Siefring's intersection theory.  Fix a trivialization $\Phi$
of $\xi$ along~$\gamma_\infty$.  The first observation is that
since $d_0(u^\pm) \le c_N(u^\pm) = 0$, the eigenfunctions
$e_1(u^\pm)$ both have maximal winding $\a_-^\Phi(\gamma_\infty)$.
Since $\cz{\gamma_\infty}{\Phi}$ is even, Proposition~\ref{spectral} implies
that there is only a $1$-dimensional eigenspace 
$$
E_\lambda \subset \mc{C}^\infty(\gamma_\infty^*\xi)
$$
of $A_{\gamma_\infty}$ with negative eigenvalue and winding 
$\a_-^\Phi(\gamma_\infty)$, and $e_1(u^\pm) \in E_\lambda$.
Denoting the trivial cylinder over $\gamma_\infty$ by $\RR \times \gamma_\infty$,
this implies
$$
u^\pm * (\RR \times \gamma_\infty) = 0.
$$
Indeed, there are no geometric intersections between $u^\pm$ and
$\RR \times \gamma_\infty$ since $u^\pm_M(\CC)$ is the interior of~$S_\pm$,
but one must still rule out asymptotic contributions, i.e.~``hidden''
intersections at infinity.  These are
characterized in \cite{Siefring:intersection} in terms of relative winding
numbers, and in the present situation, the asymptotic contribution to
$u^\pm * (\RR \times \gamma_\infty)$ vanishes if and only if the
asymptotic representative describing the approach of $u^\pm$ to
$\RR \times \gamma_\infty$ at infinity has maximal winding.  This is true
since $\wind^\Phi(e_1(u^\pm)) = \a_-^\Phi(\gamma_\infty)$.

Now suppose $u : \dot{\Sigma} \to \RR \times M$ is a $J_+$-holomorphic
curve with the specified properties.  We claim
\begin{equation}
\label{eqn:noIntersection}
u * u^\pm = 0.
\end{equation}
To see this, first use $\RR$-translation to move $u^\pm$ until its image is
contained in $[0,\infty) \times M$, which is possible since $u^\pm$ has
no negative punctures.  Then notice that since $u$ has only one positive 
puncture and it is asymptotic to $\gamma_\infty$, $u$ admits a homotopy 
through asymptotically cylindrical (but not necessarily $J_+$-holomorphic) 
maps to a map whose intersection with $[0,\infty) \times M$ is identical
to the top half of the trivial cylinder $\RR \times \gamma_\infty$.  Using the homotopy
invariance of the intersection product, it follows that $u * u^\pm =
u^\pm * (\RR \times \gamma_\infty) = 0$.

Finally, let $e_1(u)$ denote the nontrivial asymptotic eigenfunction
in \eqref{eqn:asympFormula} that controls the approach of $u$ to 
$\gamma_\infty$ at its positive puncture.  If $e_1(u) \not\in E_\lambda$,
then $\wind^\Phi(e_1(u)) < \a_-^\Phi(\gamma_\infty) = \wind^\Phi(e_1(u^\pm))$.
In this case the projections of $u$ and $u^\pm$ to $M$ obviously
intersect each other near~$\gamma_\infty$, implying that some $\RR$-translation
of $u$ intersects $u^\pm$, but this is impossible by \eqref{eqn:noIntersection}.
We therefore have $e_1(u) \in E_\lambda$.  But observe now that
applying $\RR$-translations to $u$ modifies $e_1(u)$ by multiplication with
a positive constant, so since $\dim E_\lambda = 1$ and $e_1(u^+)$ and
$e_1(u^-)$ have opposite signs, there exists a unique $\RR$-translation 
for which $e_1(u)$ precisely matches either $e_1(u^+)$ or~$e_1(u^-)$.
For concreteness, suppose $e_1(u) = e_1(u^+)$.  Then the main results of
\cite{Siefring:asymptotics} imply that unless $u = u^+$ up to
parametrization, there is a nontrivial asymptotic eigenfunction
controlling the approach of $u$ to $u^+$ at their positive ends, and it lies 
in a different eigenspace, with winding strictly less 
than~$\a_-^\Phi(\gamma_\infty)$.  The characterization of asymptotic
contributions in \cite{Siefring:intersection} then implies that
$u * u^+ > 0$, again contradicting \eqref{eqn:noIntersection}.
\end{proof}

\begin{remark}
The above lemma does not specifically require the manifold $M$ to be reducible:
it only requires the existence of a simple Reeb orbit $\gamma_\infty$
spanned by two disjoint embedded index~$1$ holomorphic planes $u^\pm$ that
approach $\gamma_\infty$ ``from opposite sides'' in the sense described above.
The conditions $c_N(u^\pm) = u^\pm * u^\pm = u^+ * u^- = 0$ follow
automatically from these assumptions via \eqref{cneq} and Siefring's
adjunction formula \eqref{eqn:adjunction}.
\end{remark}

\subsection{Overtwisted contact $3$-manifolds}
\label{sec:OTcase}

If $(M,\xi)$ is overtwisted, then Eliashberg's appendix to
\cite{Yau:overtwisted} uses the following geometric picture 
to prove vanishing of contact homology.  There is a nondegenerate contact
form $\alpha_+$ and an almost complex structure $J_+ \in \jJ(\alpha_+)$
admitting an embedded Fredholm regular $J_+$-holomorphic plane
$$
u^\infty = (u^\infty_\RR,u^\infty_M) : \CC \to \RR \times M
$$
with index~$1$, asymptotic to a simple Reeb orbit
$$
\gamma_\infty : S^1 \to M
$$
with even Conley-Zehnder index, such that $u^\infty$ is (up to
parametrization and $\RR$-translation) the only nontrivial 
$J_+$-holomorphic curve in $\RR \times M$ with one positive puncture asymptotic to $\gamma_\infty$
(and arbitrary negative punctures).  A more detailed version of this
construction can be extracted as a special case from \cite{Wendl:openbook2},
using intersection-theoretic arguments similar to those of
\S\ref{sec:reducible} above.  Since $u^\infty$ is asymptotic to a simple
orbit, Lemma~\ref{lemma:dinfty} implies $\dinfty(u^\infty) = 
\bar{\sigma}(u^\infty) - 1 = 0$, and since it is also embedded,
$\dtotal(u^\infty)=0$.  Moreover, by \eqref{cneq}, $c_N(u^\infty)=0$,
so the adjunction formula \eqref{eqn:adjunction} now gives
$$
u^\infty * u^\infty = 2\dtotal(u^\infty) + c_N(u^\infty) +
\left[ \bar{\sigma}(u^\infty) - 1 \right] = 0,
$$
implying that $u^\infty$ is nicely embedded.

\section{A local adjunction formula for breaking holomorphic annuli}
\label{sec:localAdjunction}

The aim of this section is to prove Theorem \ref{thm:localAdjunction} and derive Corollary \ref{cor:localAdjunction}. 
We assume throughout that $M$ is a $3$-manifold
endowed with $\mc{C}^\infty$-converging sequences 
of contact forms $\alpha_k \to \alpha_\infty$ and admissible almost complex
structures $J_k \in \jJ(\alpha_k)$, $k \le \infty$.
Fix a nondegenerate closed Reeb orbit $\gamma : S^1 \to M$ for~$\alpha_\infty$
with covering multiplicity $m \in \NN$, period $T > 0$ and parity
$p(\gamma) \in \{0,1\}$.
We consider a sequence of $J_k$-holomorphic annuli
$$
u_k : ([-R_k,R_k'] \times S^1,i) \to (\RR \times M,J_k),
$$
where $R_k,R_k' \to \infty$ and $u_k$ converges in the SFT topology to 
a broken $J_\infty$-holomorphic annulus
$$
u_k \to (u_\infty^+ | u_\infty^-)
$$
in which both levels are embedded and $\gamma$ is the breaking orbit.
More precisely, $u_\infty^\pm$ are embedded $J_\infty$-holomorphic half-cylinders
$$
u_\infty^- : [0,\infty) \times S^1 \to \RR \times M, \qquad
u_\infty^+ : (-\infty,0] \times S^1 \to \RR \times M
$$
with 
$$
u_\infty^\pm(s,t) = \exp_{(Ts,\gamma(t))} h_\pm(s,t)
$$
for some translation-invariant metric on $\RR \times M$ and asymptotic
representatives satisfying $h_\pm(s,\cdot) \to 0$ with all derivatives
as $s \to \mp\infty$.  To say what the convergence $u_k \to (u_\infty^+ | u_\infty^-)$
means, denote the $\RR$-translation action on $\RR \times M$ by
$$
\tau_c : \RR \times M \to \RR \times M : (r,x) \mapsto (r + c,x)
$$
for $c \in \RR$.  Then we require
$$
\tau_{r_k} \circ u_k(\cdot + R_k',\cdot) \to u_\infty^+ \quad \text{ in } \quad
\Cinftyloc((-\infty,0] \times S^1,\RR \times M)
$$ 
for some sequence $r_k \to -\infty$, while
$$
\tau_{r_k} \circ u_k(\cdot - R_k,\cdot) \to u_\infty^- \quad\text{ in }\quad
\Cinftyloc([0,\infty) \times S^1,\RR \times M)
$$ 
for some sequence $r_k \to +\infty$.  Additionally, choose diffeomorphisms 
$\varphi_- : [-1,0) \to [0,\infty)$ and $\varphi_+ : (0,1] \to (-\infty,0]$,
let $\pi_M : \RR \times M \to M$ denote the natural projection, and define
the continuous map
$$
\bar{u}_\infty^M : [-1,1] \times S^1 \to M : (s,t) \mapsto
\begin{cases}
\pi_M \circ u_\infty^+(\varphi_+(s),t) & \text{ if $s > 0$},\\
\gamma(t) & \text{ if $s=0$},\\
\pi_M \circ u_\infty^-(\varphi_-(s),t) & \text{ if $s < 0$}.
\end{cases}
$$
We then also require the existence of a sequence of
diffeomorphisms $\varphi_k : [-1,1] \times S^1 \to [-R_k,R_k'] \times S^1$ 
such that
$$
\pi_M \circ u_k \circ \varphi_k \to \bar{u}_\infty^M \quad\text{ in }\quad
C^0([-1,1] \times S^1,M).
$$
With these hypotheses understood, the statement of 
Theorem~\ref{thm:localAdjunction} is that for all $k$ sufficiently large,
$$
2\delta(u_k) = 2[ \delta_\infty(u_\infty^+) + \delta_\infty(u_\infty^-) ]
+ \left[ \bar{\sigma}_+(\gamma) - 1 \right] 
 + \left[ \bar{\sigma}_-(\gamma) - 1 \right]
+ (m - 1) p(\gamma).
$$

The proof is based on a relative adjunction formula in the style of
Hutchings \cite{Hutchings:index}.  Since $u_\infty^+$ and $u_\infty^-$
are embedded, we are free to assume $u_k$ is embedded
near the boundary of its domain for sufficiently large~$k$; moreover,
if we reparametrize $u_\infty^\pm$ by suitable shifts to focus only on
neighborhoods of $\pm\infty$, the corresponding adjustments in $u_k$
can be arranged so that its tangent spaces are close to those of a
trivial cylinder for large~$k$.  This means replacing the domains
$[-R_k,R_k'] \times S^1$ of $u_k$ with smaller domains that nonetheless
still expand to infinite length, and we do not lose any singularities this
way since $u_\infty^\pm$ are both embedded on the corresponding portions
of their domains that are being discarded.  With this understood, if
we choose a trivialization $\Phi$ of 
$\xi = \ker \alpha_\infty$ along~$\gamma$, this determines
a trivialization of the normal bundle of $u_k$ along its boundary
uniquely up to homotopy
for large~$k$.  Define the \defin{relative self-intersection number}
of~$u_k$,
$$
u_k \bullet_\Phi u_k \in \ZZ
$$
as the algebraic count of intersections between $u_k$ and a generic
perturbation of $u_k$ that is pushed in the direction of $\Phi$ near
the boundary.  This number depends on~$\Phi$ up to homotopy.  We can
similarly define $u_\infty^\pm \bullet_\Phi u_\infty^\pm$ with the condition
that the second copy of $u_\infty^\pm$ is pushed by some small but nonzero
amount in the direction of $\Phi$ both near its boundary and near infinity.
The convergence $u_k \to (u_\infty^+ | u_\infty^- )$ then implies
\begin{equation}
\label{eqn:selfIntSum}
u_k \bullet_\Phi u_k = u_\infty^+ \bullet_\Phi u_\infty^+ +
u_\infty^- \bullet_\Phi u_\infty^-
\end{equation}
for sufficiently large~$k$.  The relative adjunction formula relates
these self-intersection numbers to the count of double points and
corresponding relative first Chern numbers
$$
c_1^\Phi(u_k^*T(\RR \times M)), \, c_1^\Phi((u_\infty^\pm)^*T(\RR \times M)) \in \ZZ,
$$
defined by regarding $\Phi$ as a trivialization of the normal bundle
of $u_k$ or $u_\infty^\pm$ over the boundary and/or near infinity---this
sums with the canonical parallelization of the domains (annuli and half-cylinders)
to give trivializations of the pulled back tangent bundle on these regions.
Appealing again to the convergence $u_k \to (u_\infty^+ | u_\infty^- )$, we
have
\begin{equation}
\label{eqn:c1sum}
c_1^\Phi(u_k^*T(\RR \times M)) = c_1^\Phi((u_\infty^+)^*T(\RR \times M)) +
c_1^\Phi((u_\infty^-)^*T(\RR \times M))
\end{equation}
for large~$k$.  Note that if $u_k$ is immersed with normal bundle $N_{u_k}$, then
$c_1^\Phi(u_k^*T(\RR \times M)) = c_1^\Phi(N_{u_k})$ since the domains have
vanishing Euler characteristic.  If $u_k$ is not immersed, one can
perturb it to an immersion without changing $c_1^\Phi(u_k^*T(\RR \times M))$,
so the same argument used to prove the adjunction formula for closed
holomorphic curves gives the relative formula
$$
u_k \bullet_\Phi u_k = 2\delta(u_k) + c_1^\Phi(u_k^*T(\RR \times M)).
$$
This makes crucial use of the fact that $u_k$ is embedded near
the boundary for large~$k$, so pushing it in normal directions determined
by $\Phi$ does not produce any intersections near the boundary.  

The analogous story for $u_\infty^\pm$ is slightly more complicated if 
$\gamma$ has covering multiplicity $m > 1$, because new intersections 
\emph{can} appear
near infinity after pushing off via~$\Phi$.  This phenomenon was observed
in \cite{Hutchings:index}*{\S 3.2} and quantified in terms of the \emph{writhe}
of a braid determined by the asymptotic behaviour of~$u_\infty^\pm$.
Using notation adapted from \cite{Siefring:intersection}*{\S 3.2}
(see also \cite{Wendl:Durham}*{\S 4.3}), we denote by
$$
i_\infty^\Phi(u_\infty^\pm) \in \ZZ
$$
the algebraic count of intersections near infinity between $u_\infty^\pm$ and a
small perturbation of itself via~$\Phi$.
This number matches the writhe described in \cite{Hutchings:index} up
to a sign.
It is only nonzero if $m \ge 2$, and in that case it depends
on $\Phi$ up to homotopy; in \cite{Siefring:intersection} it is expressed
as a sum of winding numbers of asymptotic eigenfunctions that 
control the relative approach of different parametrizations of $u_\infty^\pm$
near the orbit.  The bounds on these winding numbers coming from
Proposition~\ref{spectral} lead to the bound
\begin{equation}
\label{eqn:iInftyBound}
i_\infty^\Phi(u_\infty^\pm) \ge \Omega_\mp^\Phi(\gamma) :=
\pm (m-1) \a_\pm^\Phi(\gamma) + \left[ \bar{\sigma}_\pm(\gamma) - 1 \right],
\end{equation}
which furnishes the definition of the number $\dinfty(u_\infty^\pm)$ counting
``hidden'' double points at infinity:
\begin{equation}
\label{eqn:dinftyDef}
\dinfty(u_\infty^\pm) = \frac{1}{2}\left[ i_\infty^\Phi(u_\infty^\pm) - \Omega_\mp^\Phi(\gamma) \right] \ge 0.
\end{equation}
Including the contribution from intersections near infinity,
the relative adjunction formula for $u_\infty^\pm$
takes the form
$$
u_\infty^\pm \bullet_\Phi u_\infty^\pm = 2\delta(u_\infty^\pm) +
c_1^\Phi((u_\infty^\pm)^*T(\RR \times M)) + i_\infty^\Phi(u_\infty^\pm).
$$
We are now ready to prove both the theorem and the corollary.

\begin{proof}[Proof of Theorem~\ref{thm:localAdjunction}]
We can use the various relative adjunction formulas to rewrite both the left
and right hand sides of \eqref{eqn:selfIntSum}, thus
\begin{equation*}
\begin{split}
2\delta(u_k) + c_1^\Phi(u_k^*T(\RR \times M)) &= 2\delta(u_\infty^+) +
c_1^\Phi((u_\infty^+)^*T(\RR \times M)) + i_\infty^\Phi(u_\infty^+) \\
&+ 2 \delta(u_\infty^-) +
c_1^\Phi((u_\infty^-)^*T(\RR \times M)) + i_\infty^\Phi(u_\infty^-).
\end{split}
\end{equation*}
The terms $\delta(u_\infty^\pm)$ vanish since $u_\infty^+$ and $u_\infty^-$
are embedded, and combining this with \eqref{eqn:c1sum} gives
$$
2\delta(u_k) = i_\infty^\Phi(u_\infty^+) + i_\infty^\Phi(u_\infty^-).
$$
Now plugging in \eqref{eqn:dinftyDef} and $p(\gamma) = \a_+^\Phi(\gamma) -
\a_-^\Phi(\gamma)$ gives the stated formula.
\end{proof}

\begin{proof}[Proof of Corollary~\ref{cor:localAdjunction}]
The local adjunction formula implies that if $\d(u_k)=0$, then
\begin{enumerate}
\item $\delta_\infty(u_\infty^+) = \delta_\infty(u_\infty^-) = 0$,
\item $\bar{\sigma}_+(\gamma) = \bar{\sigma}_-(\gamma) = 1$, and
\item $\gamma$ is either simply covered or has even parity.
\end{enumerate}
In the case with even parity, we can derive further constraints on the
multiplicity $m$ from the condition on the spectral covering numbers: recalling
Proposition~\ref{spectral}, the extremal winding numbers $\a_\pm^\Phi(\gamma)$
for positive and negative eigenvalues match, and $\bar{\sigma}_\pm(\gamma) = 1$ 
means that any nontrivial eigenfunction in the corresponding
eigenspaces $E_\pm$ is simply covered.  Both of these eigenspaces are also
$1$-dimensional, and since $\gamma$ has multiplicity~$m$, there is a linear 
$\ZZ_m$-action on each $E_\pm$ generated 
by the map that sends an eigenfunction $e \in E_\pm$ to $e(\cdot + 1/m)$.
This defines a real $1$-dimensional representation of $\ZZ_m$, and the
representation must be faithful since $E_\pm$ contains simply covered
eigenfunctions.  Real $1$-dimensional representations of finite groups
can act only by $\pm 1$, so the only possibility for $m > 1$ is that $m=2$
and the generator of $\ZZ_2$ acts by~$-1$.  We claim finally that in this
case, the underlying simple orbit has odd parity.  Indeed, Proposition~\ref{spectral}
implies that it would otherwise have two eigenfunctions with the same winding
but eigenvalues of opposite sign, and $E_\pm$ would then have to consist of the
double covers of those eigenfunctions, which is a contradiction.
\end{proof}

\section{Compactness for nicely embedded planes in cobordisms}
\label{sec:compactness}

In this section we fix the following assumptions.  Let $(W,d\lambda)$ denote
a $4$-dimensional Liouville cobordism with concave boundary 
$(M_-,\xi_- = \ker\alpha_-)$ and convex boundary $(M_+,\xi_+ = \ker\alpha_+)$, where $\lambda|_{TM_\pm} = 
\alpha_\pm$, with $\alpha_-$ assumed nondegenerate and $\alpha_+$ Morse-Bott.
The symplectic completion of $W$ will be denoted as usual by $\overline{W}$, and
we choose $J \in \jJ(W,\omega,\alpha_+,\alpha_-)$ to be generic in the
interior of $W$ such that its restriction $J_-$ to the negative end is also generic;
in particular, this means that all simple $J$-holomorphic curves
in $\overline{W}$ that pass through the interior of $W$ have nonnegative index,
and all simple $J_-$-holomorphic curves in $\RR \times M_-$ 
other than trivial cylinders have index
at least~$1$.  Fix a simply covered Reeb orbit
$$
\gamma_\infty : S^1 \to M_+.
$$
The main objective of this section is the following theorem, which 
characterizes the closure in the
SFT compactification of the set of planes in $\mM(J,\gamma_\infty,\emptyset)$
that are nicely embedded in the sense of Definition~\ref{ne}.
For application to our main theorems,
the results of \S\ref{sec:seed} permit us to ignore holomorphic buildings
with nontrivial upper levels.

\begin{thm}
\label{thm:compactness}
Suppose $u_k \in \mM(J,\gamma_\infty,\emptyset)$ is a sequence of nicely
embedded planes converging in the sense of \cite{SFTcompactness} to a
holomorphic building $u_\infty \in \overline{\mM}(J,\gamma_\infty,\emptyset)$
with no nontrivial upper levels but at least one nontrivial lower level.
Then all components of the levels of $u_\infty$ other than trivial cylinders
are nicely embedded, all breaking orbits are either simply covered or are
doubly covered bad orbits, and
$u_\infty$ fits one of the following descriptions (see Figure~\ref{fig:types}):
\begin{itemize}
\item Type~(I): $(v_0 | v_1^-)$, where $v_0$ is an index~$0$ cylinder,
$v_1^-$ is an index~$1$ plane, and the breaking orbit has even parity.
\item Type~(II): $(v_0 | v_1^-)$, where $v_0$ is an index~$0$ cylinder,
$v_1^-$ is an index~$2$ plane, and the breaking orbit has odd parity.
\item Type~(III): $(v_0 | v_1^-)$, where $v_0$ is an index~$1$ cylinder,
$v_1^-$ is an index~$1$ plane, and the breaking orbit has even parity.
\item Type~(IV): $(v_0 | v_1^-)$, where $v_0$ has index~$0$ and two negative
punctures, $v_1^-$ is a disjoint union of two index~$1$ planes, and
both breaking orbits have even parity.
\item Type~(V): $(v_0 | v_1^- | v_2^-)$, where $v_0$ has index~$0$ and two
negative punctures, $v_1^-$ is the disjoint union of a trivial cylinder with
an index~$1$ plane, and $v_2^-$ is an additional index~$1$ plane, with all
breaking orbits having even parity.
\item Type~(VI): $(v_0 | v_1^- | v_2^-)$, where $v_0$ is an index~$0$ cylinder,
$v_1^-$ is an index~$1$ cylinder and $v_2^-$ is an index~$1$ plane, 
the breaking orbit between
$v_0$ and $v_1^-$ has odd parity, and the breaking orbit between
$v_1^-$ and $v_2^-$ has even parity.
\end{itemize}
\end{thm}

\begin{figure}
\includegraphics{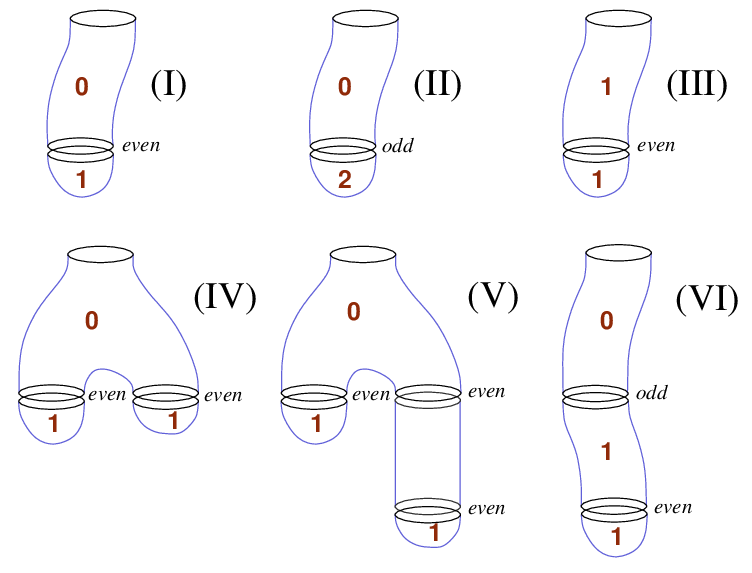}
\caption{\label{fig:types} The six types of holomorphic buildings in
Theorem~\ref{thm:compactness}.}
\end{figure}

We begin with a few preliminary observations that will be used repeatedly
in the proof.  Recall that by Prop.~\ref{prop:tree}, $u_\infty$ must
have the structure of a tree, with all connected components having exactly one
positive puncture and the bottom level being a disjoint union of capping
planes.  The top of this tree is the main level, which will always be 
somewhere injective since $\gamma_\infty$ is a simply covered orbit.
Moreover, since $J$ is generic, Equation~\ref{cneq} and
Lemma~\ref{lemma:neIndex} imply that the curves $u_k$ in our sequence can 
be assumed to satisfy either $\ind{u_k} \in \{1,2\}$ and $c_N(u_k)=0$
or $\ind{u_k}=0$ and $c_N(u_k) = -1$.
Note that all the building types in the above theorem
have total index~$2$ except for Type~(I), which occurs in the index~$1$ case.
We shall denote
$$
u_\infty = (v_0 | v_1^- | \ldots | v_N^-)
$$
where by assumption $N \ge 1$, and each $v_j^-$ is in general a disjoint
union of $m_j \ge 1$ connected curves
$$
v_{j,1}^-,\ldots,v_{j,m_j}^-,
$$
each having exactly one positive puncture.  
The definition of the normal Chern number, together with the relation 
\eqref{eqn:parity} between parities and winding numbers, implies the formula
$$
0 \ge c_N(u_k) = c_N(v_0) + \sum_{j=1}^N \sum_{i=1}^{m_j} c_N(v_{j,i}^-) + 
\sum_{j=1}^N \sum_{i=1}^{m_j} p(\gamma_{j,i}),
$$
where $\gamma_{j,i}$ denotes the asymptotic orbit at the unique positive
puncture of $v_{j,i}^-$, i.e.~the $\gamma_{j,i}$ are all the breaking orbits
of~$u_\infty$.  Since $u_\infty$ has no negative punctures in its lowest
level, we can conveniently repackage this formula as
\begin{equation}
\label{eqn:cnSum}
0 \ge c_N(u_k) = \hat{c}_N(v_0) + \sum_{j=1}^N \sum_{i=1}^{m_j}
\hat{c}_N(v_{j,i}^-),
\end{equation}
where for any punctured holomorphic curve~$w$, we define
$\hat{c}_N(w)$ to be the sum of $c_N(w)$ with the parities of the asymptotic
orbits at all its negative punctures.

\begin{lem}
\label{lemma:even}
All the components $v_{j,i}^-$ in lower levels have $c_N(v_{j,i}^-) = 0$,
and one of the following holds:
\begin{enumerate}
\item All breaking orbits in $u_\infty$ have even parity and the main
level $v_0$ satisfies $c_N(v_0)=0$, or
\item The main level $v_0$ is a cylinder with $c_N(v_0)=-1$ whose negative 
asymptotic orbit has odd parity, and all other breaking orbits have even parity.
\end{enumerate}
\end{lem}
\begin{proof}
Genericity implies $\ind{v_0} \ge 0$, so by \eqref{cneq}, $c_N(v_0) \ge -1$
with equality if and only if all the asymptotic orbits of $v_0$ have
odd parity, hence $\hat{c}_N(v_0) \ge 0$.
For components $v_{j,i}^-$ in lower levels, \eqref{cnineq}
implies $c_N(v_{j,i}^-) \ge 0$ unless $v_{j,i}^-$ is a cover of a trivial cylinder.
If on the other hand $v_{j,i}^-$ covers a trivial cylinder over some orbit~$\gamma$,
Proposition~\ref{prop:trivCyls} gives $\ind{v_{j,i}^-} \ge 0$, with strict
inequality unless $\gamma$ is elliptic or the cover is unbranched.
Strict inequality would imply $c_N(v_{j,i}^-) \ge 0$ due to
\eqref{cneq}.  The scenario $c_N(v_{j,i}^-) < 0$ can thus happen only if
$\gamma$ is elliptic and $\ind{v_{j,i}^-} = 0$, in which case \eqref{cneq} gives
$c_N(v_{i,j}) = -1$ and thus $\hat{c}_N(v_{i,j}) \ge 0$, with strict
inequality if $v_{i,j}$ has more than one negative puncture.
Using \eqref{eqn:cnSum}, we conclude that $\hat{c}(v_0)=0$ and
$\hat{c}(v_{j,i}^-)=0$ for every $i,j$, which implies that all components
in $u_\infty$ having negative normal Chern number have exactly one
negative puncture, i.e.~they are cylinders.  For components in lower levels,
this means they are trivial cylinders over elliptic Reeb orbits, and no
other lower level components can have any negative asymptotic orbits with
odd parity since this would imply $\hat{c}_N(v_{j,i}^-) > 0$.  Note that
since $u_\infty$ is assumed to be \emph{stable} in the sense of
\cite{SFTcompactness}, it does not have any levels consisting only of trivial
cylinders.  It follows that if $u_\infty$ has any odd breaking orbits at all, 
then $v_0$ is an index~$0$ cylinder with $c_N(v_0)=-1$ and a negative orbit
with odd parity, but all other breaking orbits in the building are even.
\end{proof}

\begin{proof}[Proof of Theorem~\ref{thm:compactness}]
Consider first the case where all breaking orbits have even parity.
We know that $v_0$ is somewhere injective since $\gamma_\infty$ is a simple
orbit, but each lower level component $v_{j,i}^-$ could be a multiple
cover, say a $k_{j,i}$-fold cover of a somewhere injective curve~$w_{j,i}$.
Each $w_{j,i}$ that is not a trivial cylinder satisfies $\ind{w_{j,i}} \ge 1$
due to genericity, so by Proposition~\ref{prop:coversHyp},
$$
\ind{v_{j,i}^-} \ge k_{j,i} \cdot \ind{w_{j,i}} \ge k_{j,i},
$$
with strict inequality unless the cover is unbranched.  Thus if any $k_{j,i}$
is greater than~$1$, the fact that $\ind{u_k} \le 2$ then implies
$\ind{v_{j,i}^-} = k_{j,i}=2$, $v_{j,i}^-$ is an unbranched double cover
of~$w_{j,i}$, and all other components in the building must 
have index~$0$.  Note that if this happens, $v_{j,i}^-$ cannot be a plane,
as $w_{j,i}$ would then also be a plane and the Riemann-Hurwitz formula
precludes the existence of an \emph{unbranched} cover $\CC \to \CC$.
But $u_\infty$ definitely also has components that are capping planes,
which also must have positive index, so this gives a contradiction,
proving that every $v_{j,i}^-$ is either somewhere injective or is a trivial
cylinder over an orbit with even parity.   Since the curves $u_k$ converging
to $u_\infty$ are all embedded, Corollary~\ref{cor:localAdjunction} now
implies that all breaking orbits are either simple or doubly covered bad orbits.
To see that all components are also nicely embedded, we use
Proposition~\ref{prop:intBuilding} to write
$$
0 \ge u_k * u_k \ge v_0 * v_0 + \sum_{j=1}^N \sum_{i,\ell=1}^{m_j}
v_{j,i}^- * v_{j,\ell}^-.
$$
Positivity of intersections together with Prop.~\ref{prop:evenCylinder}
implies that all terms on the right hand side are nonnegative, hence all of
them vanish, including each $v_{j,i}^- * v_{j,i}^-$.  Since $c_N(v_{j,i}^-)
= c_N(v_0) = 0$, the adjunction formula \eqref{eqn:adjunction} then gives $\dtotal(v_{j,u}^-) =
\dtotal(v_0) = 0$, hence all components are nicely embedded.

By Lemma~\ref{lemma:even}, we must also consider the case where
$v_0$ is a cylinder with $c_N(v_0) = -1$ whose asymptotic orbits are both odd,
while all other breaking orbits are even and all lower level components
satisfy $c_N(v_{j,i}^-) = 0$.  The adjunction formula \eqref{eqn:adjunction} implies
$v_0 * v_0 \ge -1$.  For any pair of components $v_{j,i}^-$ and $v_{j,\ell}^-$
that are not both trivial cylinders, positivity of intersections gives
$v_{j,i}^- * v_{j,\ell}^- \ge 0$; note that this is guaranteed even if they
are (covers of) the same curve since $\RR$-invariance allows us to push
one of them so that they have only isolated intersections.  This trick
only fails if both are covers of the same trivial cylinder, but in that case,
we know that the underlying Reeb orbit cannot have odd parity, hence
Prop.~\ref{prop:evenCylinder} applies again to give $v_{j,i}^- * v_{j,\ell}^- \ge 0$.
Now Proposition~\ref{prop:intBuilding} implies
$$
0 \ge u_k * u_k \ge -1 + m(\gamma),
$$
where $\gamma$ is the odd breaking orbit between $v_0$ and $v_{1,1}^-$ and
$m(\gamma) \in \NN$ denotes its covering multiplicity; we conclude that this
orbit must be simply covered.  It follows that $v_{1,1}^-$ is a somewhere
injective curve and thus has $\ind{v_{1,1}^-} \ge 1$.  For all curves in
levels below this, the absence of odd orbits means that we can again use the
arguments of the previous paragraph to rule out multiple covers, and we
conclude again that all components in $u_\infty$ are somewhere injective
except possibly for trivial cylinders.  The constraints on multiplicities
of the breaking orbits and the conclusion that all nontrivial components
are nicely embedded now follow from the same arguments using the local
adjunction formula and Proposition~\ref{prop:intBuilding}.

To obtain the classification of buildings stated in the theorem, it remains
only to add up Fredholm indices, using the fact that $\ind{v_0} \ge 0$ and
all nontrivial components in lower levels have index at least~$1$.  The
conclusions about parities of Reeb orbits then follow directly from the
index formula.
\end{proof}

\section{Proofs of the main theorems}
\label{sec:proofs}

We shall now prove Theorems~\ref{thm:main} and~\ref{thm:nonprime} in reverse
order.

\begin{proof}[Proof of Theorem~\ref{thm:nonprime}\eqref{item:OT}]
Assume $(W,d\lambda)$ is a $4$-dimensional Liouville co\-bor\-dism from 
$(M,\xi)$ to $(M_+,\xi_+)$, where $(M_+,\xi_+)$ is overtwisted, and
$\alpha$ is a nondegenerate contact form for~$\xi$.  After possibly rescaling $\alpha$
by a positive constant, we can arrange $\lambda|_{TM} = \alpha$ and
$\lambda|_{TM_+} = \alpha_+$, where $\alpha_+$ is the particular nondegenerate
contact form described in \S\ref{sec:OTcase}.  Arguing by contradiction,
assume $\alpha$ does not admit any unknotted Reeb orbit with
Conley-Zehnder index~$2$ and self-intersection number~$-1$.
By Lemma~\ref{lemma:selfLinking}, this means that nicely embedded planes
with index~$2$ and simply covered asymptotic orbits cannot exist in
$(\RR \times M,J_-)$ for any $J_- \in \jJ(\alpha)$; we will use this to exclude
some of the possible buildings listed in Theorem~\ref{thm:compactness}
and thus derive a contradiction.

Choose 
$J \in \jJ(W,d\lambda,\alpha_+,\alpha)$ to be generic in the interior of
$W$ such that its restriction to the negative cylindrical end is a 
generic element $J_- \in \jJ(\alpha)$ and its restriction to the positive
end matches $J_+ \in \jJ(\alpha_+)$ from \S\ref{sec:OTcase}.
Then the nicely embedded plane $u^\infty \in \mM(J_+,\gamma_\infty,\emptyset)$
constructed in that section gives rise to a $1$-parameter family of
nicely embedded curves in $\mM(J,\gamma_\infty,\emptyset)$, living in the
cylindrical end $[0,\infty) \times M_+ \subset \overline{W}$; we shall
refer to these henceforth as the \emph{seed curves} in~$\overline{W}$.
They are Fredholm regular by Prop.~\ref{prop:automatic}.  Let
$$
\mM\nice(J) \subset \mM(J,\gamma_\infty,\emptyset)
$$
denote the set of all nicely embedded planes in $\mM(J,\gamma_\infty,\emptyset)$
that belong to the same relative homology class as the seed curves.
These all have index~$1$, and the uniqueness of curves in $(\RR \times M_+,J_+)$
implies that all of them other than the seed curves intersect the region
where $J$ is generic, thus $\mM\nice(J)$ is a smooth $1$-manifold.  Its closure
$$
\overline{\mM}\nice(J) \subset \overline{\mM}(J,\gamma_\infty,\emptyset)
$$
in the SFT compactification can now be described as follows.  If $u_k \in \mM\nice(J)$
is a sequence converging to a building with a nontrivial upper level, then the uniqueness of
curves in $(\RR \times M_+,J_+)$ implies that this building is $(u^\infty |\ \emptyset)$,
where $u^\infty$ here is an upper level and the main level is empty.  This
can only be the limit if $u_k$ consists of seed curves for sufficiently large~$k$,
so a neighborhood of $(u^\infty |\ \emptyset)$ in $\overline{\mM}\nice(J)$
is homeomorphic to the interval $(-1,0]$, with $(u^\infty |\ \emptyset)$
as the boundary point.  If the limit is any building with trivial upper
levels but a nontrivial lower level, then it is described by 
Theorem~\ref{thm:compactness}, and must in fact be a building of Type~(I)
since all the others in the list have total index~$2$.  We can thus write the
limit as $(v_0 | v_1^-)$ for an index~$0$ cylinder
$v_0$ in $(\overline{W},J)$ and an index~$1$ plane $v_1^-$ in $(\RR \times M,J_-)$,
both nicely embedded.  The breaking orbit between these must be a doubly
covered bad orbit due to the assumption excluding unknotted orbits.
Proposition~\ref{prop:gluing} on gluing implies that a neighborhood of
$(v_0 | v_1^-)$ in $\overline{\mM}\nice(J)$ is also homeomorphic to 
$(-1,0]$, with the building forming the boundary point.  We've thus shown
that each connected component of $\overline{\mM}\nice(J)$ has the 
topology of a compact connected $1$-manifold with boundary, so either a
circle or a closed interval, and at least one component is of the latter
type, namely the one containing the seed curves.  We claim in fact that
$\overline{\mM}\nice(J)$ itself is compact.  In light of the above description,
this can only fail to be true if $\overline{\mM}\nice(J)$
has infinitely many connected components, in which case we can find a sequence
of nicely embedded curves $u_k \in \mM\nice(J)$ all belonging to separate
components.  But these curves are all homologous and thus satisfy a uniform
energy bound, so they have an SFT-convergent subsequence by \cite{SFTcompactness},
whose limit is in~$\overline{\mM}\nice(J)$ by definition.
We conclude that $\overline{\mM}\nice(J)$ is a compact $1$-manifold with boundary.

The crucial observation is now that since the breaking orbit
in $(v_0 | v_1^-)$ is doubly covered, there are always exactly two choices of
decoration which give two distinct elements of $\overline{\mM}(J,\gamma_\infty,\emptyset)$
having the same curves as their main and lower levels.  Indeed, these
cannot represent equivalent elements of $\overline{\mM}(J,\gamma_\infty,\emptyset)$
since the levels are both somewhere injective and thus admit no 
automorphisms that could change the decoration.  Moreover, both buildings
can be glued via Prop.~\ref{prop:gluing} to produce smooth $1$-parameter
families of somewhere injective planes in $\mM(J,\gamma_\infty,\emptyset)$,
which will be nicely embedded since they are homologous to the seed curves,
thus both buildings also belong to~$\overline{\mM}\nice(J)$.  With this
understood, let
$$
\widehat{\mM}\nice(J) = \overline{\mM}\nice(J) \big/ \sim,
$$
where the equivalence relation identifies any two buildings that have matching
levels but different decorations.  Topologically, $\widehat{\mM}\nice(J)$
is formed by gluing components of $\overline{\mM}\nice(J)$ together along
boundary points of the form $(v_0 | v_1^-)$, all of which become \emph{interior}
points in~$\widehat{\mM}\nice(J)$.  But exactly one boundary point of
$\overline{\mM}\nice(J)$ does not belong to a matching pair, namely
$(u^\infty |\ \emptyset)$, thus $\widehat{\mM}\nice(J)$ is homeomorphic to
a compact $1$-manifold with one boundary point, giving a contradiction.
\end{proof}

\begin{proof}[Proof of Theorem~\ref{thm:nonprime}\eqref{item:nonprime}]
Assume $(M,\xi)$ is a reducible tight contact $3$-mani\-fold, and fix the
contact form $\alpha_+$ and almost complex structure $J_+ \in \jJ(\alpha_+)$
described in \S\ref{sec:reducible}, so there is a homotopically nontrivial
$2$-sphere $S \subset M$ containing a simple nondegenerate Reeb orbit
$\gamma_\infty : S^1 \to M$ that divides $S$ into hemispheres 
$S_\pm = u_M^\pm(\CC)$ that are each images of nicely embedded $J_+$-holomorphic
index~$1$ planes
$$
u^\pm = (u^\pm_\RR,u^\pm_M) : \CC \to \RR \times M.
$$
By Lemma~\ref{lemma:csUniqueness}, these are the only $J_+$-holomorphic
curves up to parametrization and 
$\RR$-translation that have a single positive puncture asymptotic
to $\gamma_\infty$ (and arbitrary negative punctures).

Now pick an arbitrary nondegenerate contact form $\alpha$ on $(M,\xi)$, 
and suppose it admits no unknotted Reeb orbit with Conley-Zehnder index~$2$
and self-linking number~$-1$.  We are free to
rescale $\alpha_+$ so that $\alpha_+ = e^f \alpha$ for some 
$f : M \to (0,\infty)$ without loss of generality.  There is then a Liouville
cobordism $(W,d\lambda)$ with
$$
W = \left\{ (r,x) \in \RR \times M \ \big|\ 0 \le r \le f(x),\ x \in M \right\}, \qquad
\lambda = e^r \alpha,
$$
which inherits the contact form $\alpha$ on its concave boundary
$M_- := \{0\} \times M$ and $\alpha_+$ on its convex boundary
$M_+ := \{ (f(x),x)\ |\ x \in M \}$.  Choose a generic $J \in 
\jJ(W,d\lambda,\alpha_+,\alpha)$ that matches $J_+$ in the positive end
and has a generic restriction $J_- \in \jJ(\alpha)$ to the negative end.
The $\RR$-translations of $u^+$ and $u^-$ then give rise to a disjoint pair of
$1$-parameter families of nicely embedded seed curves in the completion
$\overline{W}$, living in the positive cylindrical end.  

As in the overtwisted case, we consider the space $\mM\nice(J) \subset
\mM(J,\gamma_\infty,\emptyset)$ of nicely embedded planes that are in the
same relative homology class with either family of seed curves, and its 
closure $\overline{\mM}\nice(J)$ in the SFT compactification.  Then
$\overline{\mM}\nice(J)$ contains the two elements $(u^\pm |\ \emptyset)$
with empty main levels, plus (by Theorem~\ref{thm:compactness}) buildings
of the form $(v_0 | v_1^-)$ with a nicely embedded index~$1$ plane $v_1^-$ in 
the lower level and breaking orbits that are doubly covered.  The same
argument as in the overtwisted case proves that $\overline{\mM}\nice(J)$
has the topology of a compact $1$-manifold with boundary, where boundary
points of the form $(v_0 | v_1^-)$ come in matching pairs with the same
levels but different decorations due to the doubly covered breaking orbit.  
The space
$$
\widehat{\mM}\nice(J) = \overline{\mM}\nice(J) \big/ \sim
$$
defined by identifying matching pairs is thus a compact $1$-manifold with
exactly two boundary points, the two buildings $(u^+ |\ \emptyset)$ and
$(u^- |\ \emptyset)$.  It follows that these two buildings belong to the
same connected component in $\widehat{\mM}\nice(J)$, and the images of
the curves or buildings in this component under the projection
$\RR \times M \to M$ then give a continuous $1$-parameter family of
disks with fixed boundary $\gamma_\infty$ forming a homotopy from $S_+$
to $S_-$ in~$M$.  This is a contradiction since $[S] \ne 0 \in \pi_2(M)$.
\end{proof}

\begin{proof}[Proof of Theorem~\ref{thm:main}]
Recall from \S\ref{sec:stdSphere} the definitions of the standard
contact form $\alpha\std$ and the almost complex structure 
$J\std \in \jJ(\alpha\std)$.
Given a Liouville cobordism $(W,d\lambda)$ from some contact manifold 
$(M,\xi)$ to $(S^3,\xi\std)$ with $\lambda|_{TS^3} = \alpha\std$ and
a nondegenerate contact form $\lambda|_{TM} = \alpha$ on $(M,\xi)$,
one can choose a generic
$J \in \jJ(W,d\lambda,\alpha\std,\alpha)$ that matches $J\std$ in the positive
cylindrical end and has a generic restriction $J_- \in \jJ(\alpha)$ to the
negative cylindrical end.  The seed curves $u_w$ constructed 
in \S\ref{sec:stdSphere} that are contained
in $[0,\infty) \times S^3$ can then equally well be regarded as
nicely embedded $J$-holomorphic planes in the completed cobordism~$\overline{W}$
with images in the positive cylindrical end.  Recall that by
Lemma~\ref{lemma:uwUniqueness}, every curve in $\RR \times S^3$ with one
positive puncture asymptotic to the particular orbit $\gamma_\infty$
(and arbitrary negative punctures) is one of these planes.

Let $\mM_1\nice(J) \subset \mM_1(J,\gamma_\infty,\emptyset)$ denote the set of
all nicely embedded planes in the same relative
homology class as the seed curves, carrying the extra data of one marked point, 
and denote its closure in the SFT
compactification by $\overline{\mM}_1\nice(J) \subset \overline{\mM}_1(J,\gamma_\infty,\emptyset)$.  
All curves in $\mM_1\nice(J)$ have index~$2$ and are Fredholm regular,
so $\mM_1\nice(J)$ is a smooth $4$-dimensional manifold.
As a consequence of
Lemma~\ref{lemma:uwUniqueness}, all buildings in $\overline{\mM}_1\nice(J)$
with nontrivial upper levels are of the form $(u_w |\ \emptyset)$ with
$u_w$ a seed curve and the main level empty.  All other buildings in
$\overline{\mM}\nice(J)$ are described by Theorem~\ref{thm:compactness}.

We will now show that if $\alpha$ admits no unknotted Reeb orbits
with self-linking number~$-1$ and Conley-Zehnder index $2$, then it must admit
one with Conley-Zehnder index~$3$.  To this end, choose points $p_+ \in S^3$
and $p_- \in M$, and a $1$-dimensional submanifold
$$
\ell(\RR) \subset \overline{W}
$$
defined via a smooth proper embedding $\ell : \RR \hookrightarrow \overline{W}$
such that $\ell(t) = (t,p_+) \in [0,\infty) \times S^3$ for all $t > 0$ sufficiently
large and $\ell(t) = (t,f(t)) \in (-\infty,0] \times M$ for $t < 0$ sufficiently
small, with $\lim_{t \to -\infty} f(t) = p_-$.  After generic perturbations
of both $p_-$ and $\ell$ away from~$+\infty$, we are free to assume:
\begin{enumerate}
\item $\ell$ is transverse to the evaluation map on the moduli space of all
somewhere injective $J$-holomorphic curves in $\overline{W}$ that are 
not fully contained in the positive end;
\item $p_-$ is not contained in any closed Reeb orbit;
\item $\RR \times \{p_-\}$ is transverse to the evaluation map on the moduli
space of all somewhere injective $J_-$-holomorphic curves in $\RR \times M$.
\end{enumerate}
Now consider the $1$-dimensional submanifold
$$
\mM\nice_\ell(J) = \ev^{-1}(\ell(\RR)) \subset \mM\nice_1(J),
$$
and for convenience define
$$
\overline{\mM}\nice_\ell(J) \subset \overline{\mM}\nice_1(J)
$$
to be the set of all buildings in its SFT-closure that have the marked
point in the main level.  The evaluation map then restricts to
$$
\overline{\mM}\nice_\ell(J) \stackrel{\ev}{\longrightarrow} \ell(\RR),
$$
and the image of this map clearly contains an interval of the form
$\ell([t_0,\infty))$ due to the seed curves.  The goal of the next two
paragraphs is to show that this map is in fact surjective onto~$\ell(\RR)$.

Observe first that the restriction of $\ev$ to $\mM\nice_\ell(J)$ is an
open map due to Prop.~\ref{prop:localFol}.  We claim moreover that any
building $u \in \overline{\mM}\nice_\ell(J)$ with $\ev(u) = \ell(t_0)$ has
a neighborhood in $\overline{\mM}\nice_\ell(J)$ whose image under $\ev$
contains $\ell([t_0,t_0 + \epsilon))$ or $\ell((t_0 - \epsilon,t_0])$ for
sufficiently small $\epsilon > 0$.  To see this, note first that unless
$u$ is a smooth curve, it is necessarily one of Types~(II) though~(VI) on the list in Theorem~\ref{thm:compactness},
but our genericity assumptions impose further restrictions: since $\ell$
intersects the evaluation map transversely, the main level of the building
must have index at least~$1$, excluding all options other than Type~(III).
We can thus write $u = (v_0 | v_1^-)$ for an index~$1$ cylinder
$v_0 \in \mM_1(J,\gamma_\infty,\gamma)$ with a marked point, and an index~$1$ 
plane $v_1^- \in \mM(J_-,\gamma,\emptyset)$, both nicely embedded.  
Here the lower level $v_1^-$ represents an isolated element in 
$\mM(J_-,\gamma,\emptyset) / \RR$, while $v_0$ has a neighborhood 
$\vV_0 \subset \mM_1(J,\gamma_\infty,\gamma)$ that is a smooth $3$-manifold,
and by choosing this neighborhood sufficiently small, we can assume that
\begin{equation}
\label{eqn:evU0}
\vV_0 \stackrel{ev}{\longrightarrow} \overline{W}
\end{equation}
has only one intersection with $\ell(\RR)$, namely at~$v_0$, and it is
transverse.  Proposition~\ref{prop:gluing} now gives a smooth gluing map
$$
\Psi : [0,\infty) \times \vV_0 \hookrightarrow \mM_1(J,\gamma,\emptyset)
$$
whose image contains all smooth curves close to $(v_0 | v_1^-)$ in the
SFT topology, and all of these belong to $\mM\nice_1(J)$.  The maps
$$
\{R\} \times \vV_0 \to \overline{W} : v \mapsto \ev(\Psi(R,v))
$$
can then be assumed to converge uniformly to \eqref{eqn:evU0} as $R \to \infty$,
implying that their algebraic count of intersections with $\ell(\RR)$ is 
$1$ for all $R > 0$ sufficiently large.  Choosing $R_0 > 0$ large and generic, 
the subset
$$
\uU_\ell := \left\{ (R,v) \in [R_0,\infty) \times \vV_0 \ \big|\ \ev(\Psi(R,v)) \in \ell(\RR) \right\}
$$
is then a smooth and properly embedded $1$-manifold that intersects 
$\{R_0\} \times \vV_0$ transversely at its boundary $\p\uU_\ell$, which is
a finite set of points.  This $1$-manifold must have at least one 
noncompact component, otherwise the algebraic count of points in $\p\uU_\ell$
could not be~$1$, hence there exists a smooth path $\uU_\ell^0 \subset 
[R_0,\infty) \times \vV_0$ whose image under $\Psi$ is a smooth family of
curves $u_t \in \mM\nice_\ell(J)$ with $t \in [0,\infty)$ such that 
$$
u_t \to (v_0 | v_1^-) \quad \text{ as } \quad t \to \infty.
$$
The image of this path under $\ev$ necessarily contains
$\ell((t_0,t_0 + \epsilon))$ or $\ell((t_0 - \epsilon,t_0))$ as claimed.

Now, given the lack of unknotted orbits with Conley-Zehnder index~$2$,
the breaking orbit in the building $(v_0 | v_1^-)$ of the previous paragraph
must be doubly covered, so that building has a twin obtained by keeping the 
same levels but changing the decoration, and the fact that both levels are
somewhere injective implies that the two buildings are not equivalent
in $\overline{\mM}_1(J,\gamma_\infty,\emptyset)$.  Thus the twin building
can also be glued using Prop.~\ref{prop:gluing} and produces a nearby
family of curves in $\mM\nice_1(J)$, some of which satisfy $\ev(u) \in \ell(\RR)$
and whose images under $\ev$ again cover an interval of the form
$\ell((t_0,t_0 + \epsilon))$ or $\ell((t_0-\epsilon,t_0))$.  But since no
two curves in $\mM\nice_1(J)$ can intersect, the curves in $\mM\nice_\ell(J)$
obtained by gluing the same curves $v_0$ and $v_1^-$ with two distinct
decorations necessarily cover two disjoint intervals, so that the entirety
of $\ell((t_0 - \epsilon,t_0 + \epsilon))$ is necessarily in the image of
$\overline{\mM}\nice_\ell(J)$ for sufficiently small $\epsilon > 0$.  
This proves that that image is an open
subset of $\ell(\RR)$, hence it is all of~$\ell(\RR)$.

With this established, we can now find a sequence $u_k \in \mM\nice_\ell(J)$
with $\ev(u_k) = \ell(t_k)$ for some sequence $t_k \to -\infty$, and a
subsequence of $u_k$ must again converge to one of the buildings listed
in Theorem~\ref{thm:compactness}, but this time with the marked point
ending up in a lower level and mapping to $\RR \times \{p_-\}$.  Since
$p_-$ is not in the image of any Reeb orbit, the marked point in the limit
does not lie on a trivial cylinder.  Transversality of $\RR \times \{p_-\}$
to the evaluation map thus implies that the component with the marked
point must have index at least~$2$, which rules out all options in the
list other than Type~(II): $u_k$ has a subsequence covergent to 
$(v_0 | v_1^-)$ where $v_1^-$ is a nicely embedded plane in $(\RR \times M,J)$
with index~$2$ and an asymptotic orbit with odd parity, which is therefore
simply covered.  This is the
promised unknotted orbit with self-linking number~$-1$ and
Conley-Zehnder index~$3$.
\end{proof}

\appendix

\section{Liouville cobordisms from exact Lagrangian caps}
\label{app:cap}

In this appendix, we provide the details behind Example~\ref{ex:cap},
using a general construction that was explained to us by Emmy Murphy.

\begin{prop}
Suppose $(M,\xi)$ is a closed contact manifold of dimension $2n-1 \ge 3$,
$\Lambda \subset M$ is a closed Legendrian submanifold and $L \subset [1,\infty) \times M$
is an exact Lagrangian cap for~$\Lambda$.  Then $L$ has an open neighbourhood
$\uU_L \subset [1,\infty) \times M$ such that, after smoothing corners,
$$
W_- := \left([0,1] \times M\right) \cup \overline{\uU}_L
$$
admits the structure of a Weinstein cobordism from $(M,\xi)$ to some 
contact manifold $(M',\xi')$, and for suitably large constants $T > 1$,
$$
W_+ := \left([1,T] \times M\right) \setminus \uU_L
$$
is a Liouville cobordism from $(M',\xi')$ to $(M,\xi)$.
\end{prop}
\begin{proof}
Being an exact Lagrangian cap means that for some choice of contact form
$\alpha$ on $(M,\xi)$ and some constant $T > 1$, the trivial Liouville cobordism
$$
(Z,d\lambda) := ([1,T] \times M, d(e^r\alpha))
$$
contains $L$ as a compact and properly embedded Lagrangian 
submanifold with $\p L = \{1\} \times \Lambda$, such that the Liouville 
vector field $\p_t$ is tangent to $L$ near $\p L$ and 
$$
\lambda|_{TL} = dg
$$
for some smooth function $g : L \to \RR$.  Note that since $\Lambda$ is 
Legendrian and $\lambda$ annihilates its dual Liouville vector field,
$g$ must be constant near~$\p L$; we shall assume without loss of generality
that it vanishes there.  By a combination of the Lagrangian and
Legendrian neighbourhood theorems, $L$ has a symplectic neighbourhood 
$(\uU_L,d\lambda)$ whose closure $\overline{\uU}_L$ is symplectomorphic to 
the unit disk bundle
in $\DD T^*L \subset T^*L$ for some choice of Riemannian metric on~$L$.
Note that this disk bundle has boundary and corners, its boundary consisting
of two smooth faces, 
$$
\p_- \overline{\uU}_L := \DD T^*L|_{\p L} \quad \text{ and } \quad
\p_+ \overline{\uU}_L := S T^*L,
$$
where $S T^*L$ is the unit cotangent bundle.
We shall write points in $T^*L$ as $(q,p)$ for $q \in L$ and $p \in T_q^*L$,
and use the metric and its induced Levi-Civita connection to
identify $T_{(q,p)}(T^*L)$ with $T_q L \oplus T_q^*L = T_q L \oplus T_q L$,
where the first splitting comes from the horizontal-vertical decomposition
given by the connection, and the second uses the isomorphism $T_q L =
T_q^*L$ determined by the metric.  The canonical Liouville form
$\lambda_0$ on $T^*L$ can then be written as
$$
\lambda_0 = - d F_0 \circ J,
$$
where $F_0(q,p) = \frac{1}{2}|p|^2$ and $J$ is the compatible
almost complex structure
on $T^*L$ that acts on $T_{(q,p)}(T^*L) = T_q L \oplus T_q L$ as
$\begin{pmatrix} 0 & \1 \\ -\1 & 0 \end{pmatrix}$.
In particular, $F_0$ is a $J$-convex function, and therefore so is
$$
F_\epsilon(q,p) := \epsilon f(q) + \frac{1}{2}|p|^2
$$
for any
smooth function $f : L \to \RR$ if $\epsilon > 0$ is sufficiently small.
Setting $\lambda_\epsilon := -d F_\epsilon \circ J$,
$d\lambda_\epsilon$ is then a symplectic form isotopic to $d\lambda_0$ on
a suitable neighbourhood of the zero-section~$L$, and since the antipodal map
$(q,p) \mapsto (q,-p)$ is $J$-antiholomorphic but preserves~$F_\epsilon$,
it also preserves the Liouville vector field $V_\epsilon$ dual 
to~$\lambda_\epsilon$, proving that $V_\epsilon$ is tangent to~$L$.

Now choose $f : L \to \RR$ in this construction to be a Morse function
that is constant with inward-pointing gradient along~$\p L$.
After possibly shrinking the neighbourhood 
$\overline{\uU}_L \cong \DD T^*L$ of~$L$, we can then assume that 
$V_\epsilon$ points transversely inward at $\p_- \overline{\uU}_L$
and transversely outward at $\p_+\overline{\uU}_L$.  Since the Liouville field
of $(Z,d\lambda)$ is also tangent to $L$ near $\p L$ and points inward
at $\{1\} \times M \subset \p Z$ (see Figure~\ref{fig:cap}), 
we can now assume after an isotopy of $\overline{\uU}_L$ 
that the two Liouville fields match near $\p_- \overline{\uU}_L$, meaning
$\lambda = \lambda_\epsilon$ on that region.  We can therefore use
$\lambda_\epsilon$ to extend $\lambda$ from $[0,1] \times M$ over $W_-$
so that the dual Liouville vector field remains gradient like, making
$W_-$ a Weinstein cobordism from $(M,\xi)$ to the new contact manifold
$(M',\xi')$, obtained by removing a neighbourhood of $\Lambda$ from
$\{1\} \times M$ and replacing it with~$S T^*L$.

It is also immediate from the above construction that $W_+$ is a strong
symplectic cobordism from $(M',\xi')$ to $(M,\xi)$, and the
exactness of the cobordism follows from the fact that $L$ is an exact
Lagrangian.  Indeed, let $\mathring{\uU}_L := \overline{\uU}_L \setminus 
\p_- \overline{\uU}_L \cong \DD T^*L|_{\mathring{L}}$.
Since $\lambda$ and $\lambda_\epsilon$ match near $\p_-\overline{\uU}_L$ and
are both primitives of the same symplectic form, $\lambda - \lambda_\epsilon$
represents an element of the compactly supported de Rham cohomology
$H^1_c(\mathring{\uU}_L)$, which is isomorphic to $H^1_c(\mathring{L})$.
But under restriction to~$L$, $\lambda_\epsilon$ vanishes and $\lambda$
is exact, so this cohomology class is zero, implying $\lambda = \lambda_\epsilon
+ dh$ on $\uU_L$ for some smooth function $h : \uU_L \to \RR$ that vanishes
near $\p_-\overline{\uU}_L$.  By multiplying $h$ with a suitable cutoff 
function, we can then find a Liouville
form on $W_+$ that matches $\lambda_\epsilon$ near~$L$ and matches $\lambda$
outside a neighbourhood of~$L$.
\end{proof}

\begin{remark}
\label{remark:subcritical}
If $W$ is a subcritical Weinstein filling of $(M,\xi)$, then the Weinstein
filling of $(M',\xi')$ obtained by stacking $W_-$ on top of $W$ is
never subcritical.  To see this, note that the Morse function $f : L \to \RR$
in the above proof can always be chosen to have exactly one critical point 
of index~$n$, in which case $F_\epsilon$ also has exactly one critical
point of index~$n$.  If $W$ is subcritical, this produces a handle
decomposition of $W \cup_M W_-$ that includes exactly one critical handle,
so $H_n(W \cup_M W_-) \ne 0$.
\end{remark}

\end{document}